\documentclass{amsart}
\usepackage{amsfonts}
\usepackage{latexsym}
\usepackage{hyperref}
\usepackage{amssymb}
\usepackage{amsmath}
\usepackage{enumerate}
\usepackage{latexsym,color,amsmath,amsthm,amssymb,amscd,amsfonts}
\usepackage{graphicx}
\usepackage[normalem]{ulem}

\usepackage[utf8]{inputenc}
\usepackage[T1]{fontenc}

\newcommand{\vol}[1]{\mathrm{vol}\left(#1\right)}

\newtheorem{thm}{Theorem}[section]
\newtheorem{cor}[thm]{Corollary}
\newtheorem{lemma}[thm]{Lemma}
\newtheorem{conjecture}[thm]{Conjecture}

\newtheorem{propo}[thm]{Proposition}
\theoremstyle{remark}

\begin{document}


\title[Local Liakopoulos-Meyer and the functional counterparts]{On local Liakopoulos-Meyer type inequalities and their functional counterparts}

\author[L. J. Al\'ias]{Luis J. Al\'ias}
\email{ljalias@um.es}
\address{Department of Mathematics. Universidad de Murcia, Spain}

\author[B. Gonz\'alez Merino]{Bernardo Gonz\'alez Merino}
\email{bgmerino@um.es}
\address{Department of Engineering and Technology of Computers, area of Applied Mathematics.
Universidad de Murcia, 
Spain}

\author[B. Mar\'in Gimeno]{Beatriz Mar\'in Gimeno}
\email{b.maringimeno@um.es}
\address{Department of Mathematics. University of Murcia, Spain}


\thanks{The first and third authors are partially supported by PID2021-124157NB-I00 funded by MCIN/AEI
/10.13039/501100011033/ ‘ERDF A way of making Europe’, Spain. The second and third authors are partially supported by Ministerio de Ciencia, Innovación y Universidades project PID2022-136320NB-I00/AEI/10.13039/501100011033/FEDER, UE. 
The third author is also supported by Contratos
Predoctorales FPU-Universidad de Murcia 2024, Spain}

\date{\today}

\begin{abstract}
We provide a functional Rogers-Shephard type inequality for log-concave functions on $\mathbb R^n$ and any $1$-reducible $s$-cover of $[n]$. As a consequence, we derive a sharp local Liakopoulos-Meyer type inequality for $n$-dimensional convex bodies and $1$-reducible $s$-covers of any $\sigma\subset[n]$, solving a question studied by Brazitikos, Giannopoulos, Liakopoulos in \cite{BGL18} as well as Alonso-Guti\'errez, Bernu\'es, Brazitikos, Carbery in \cite{ABBC21}.
\end{abstract}

\maketitle

\section{Introduction}

Comparing volumes of sections or projections and the volume of a convex body has been studied in detail. One of the most prominent results is due to Bollob\'as and Thomason \cite{BT95} who extended the classical Loomis-Whitney inequality \cite{LM49}. They showed that for any $n$-dimensional convex body (i.e. convex and compact set with non-empty interior) $K$ of $\mathbb R^n$ and any $s$-cover $(\sigma_1,\dots,\sigma_m)$ of $[n]:=\{1,\dots,n\}$, then
\begin{equation}\label{eq:BollThom}
\mathrm{vol}_n(K)^s \leq \prod_{i=1}^m \mathrm{vol}_{|\sigma_i|}(P_{H_{\sigma_i}}K).
\end{equation}
Above, $\mathrm{vol}_n(\cdot)$ denotes the n-dimensional volume or Lebesgue measure, $P_HK$ denotes the orthogonal projection of $K$ onto a given subspace $H$ of $\mathbb R^n$. Moreover, let $\sigma\subseteq [n]$.  We say that $(\sigma_1,\dots,\sigma_m)$ is a $s$-cover of $\sigma$ if $\sigma_i\subseteq\sigma$, for $i=1,\ldots,m$ and $|\{i:j\in\sigma_i\}|=s$ for every $j\in \sigma$. Finally, $H_\sigma=\langle \{e_i:i\in\sigma\}\rangle$, where $\{e_i\}_{i=1}^n$ denote the canonical basis of $\mathbb R^n$. From now on, if we write $\vol{K}$ instead of $\mathrm{vol}_n(K)$, one assumes that the volume is computed with same index as the dimension of the set $K$. Moreover, let $\mathcal K^n$ be the set of all $n$-dimensional convex bodies, and let $\mathcal L^n_i$ be the set of all $i$-dimensional linear subspaces contained in $\mathbb R^n$, for $1\leq i\leq n$.

Brazitikos, Giannopoulos, and Liakopoulos \cite{BGL18} proposed a systematic study of these type of inequalities when $(\sigma_1,\ldots,\sigma_m)$ is a $s$-cover of a
proper subset $\sigma$ of $[n]$. In order to distinguish these type of inequalities, we denote them by local inequalities. Alonso-Guti\'errez, Bernu\'es, Brazitikos, and Carbery \cite{ABBC21} as well as Manui, Ndiaye, and Zvavitch \cite[Thm.~4]{MNZ24} have recently proven the following sharp local Bollob\'as-Thomason type inequality
\begin{equation}\label{eq:ineqProj}
\vol{K}^{m-s}\vol{P_{H_\sigma^\bot}K}^s \leq \frac{\prod_{i=1}^m{n-|\sigma_i| \choose n-|\sigma|}}{{n\choose n-|\sigma|}^{m-s}} \prod_{i=1}^m\vol{P_{H_{\sigma_i}^\bot}K}.
\end{equation}
Above, $K$ is an $n$-dimensional convex body and $(\sigma_1,\dots,\sigma_m)$ is a $s$-cover of some $\sigma\subset[n]$.  Let us observe that there exists a functional version of \eqref{eq:ineqProj} in \cite{ABBC21}. 

One can consider the dual question to the ones in \eqref{eq:ineqProj} as well as \eqref{eq:BollThom} by replacing the projections by sections. Indeed, Liakopoulos \cite{L19} had the idea of generalizing Meyer's result \cite{M88} to $s$-covers, proving that
\begin{equation}\label{eq:Liako19}
\vol{K}^s \geq \frac{\prod_{i=1}^m|\sigma_i|!}{n!^s} \prod_{i=1}^m \vol{K\cap H_{\sigma_i}},
\end{equation}
where $K\in\mathcal K^n$ and $(\sigma_1,\dots,\sigma_m)$ is a $s$-cover of $[n]$. 

Following the same pattern as in \cite{ABBC21}, the authors proved in this regard that the following local Meyer type inequality
\begin{equation}\label{eq:ineqSectABBC}
\vol{K}^{m-s}\max_{x\in\mathbb R^n}\vol{K\cap (x+H_\sigma^\bot)}^s\geq\dfrac{\prod_{i=1}^m(|\sigma|-|\sigma_i|)!}{(n(m-s)))!}\prod_{i=1}^m\vol{K\cap H_{\sigma_i}^\bot}
\end{equation}
holds for every $K\in\mathcal K^n$ and $(\sigma_1,\dots,\sigma_m)$ being a $s$-cover of some $\sigma\subset[n]$. Let us realize that, generally, the inequalities above are not sharp. The exceptional case of $m=2$ and $s=1$ was solved in \cite{AAGJV}:
\begin{equation}\label{eq:AACJVs=1}
   \vol{K} \max_{x\in\mathbb R^n}\vol{K\cap(x+H_\sigma^\bot}  \geq \frac{|\sigma_1|!|\sigma_2|!}{|\sigma|!} \vol{K\cap H_{\sigma_1}^\bot}\vol{K\cap H_{\sigma_2}^\bot},
    \end{equation}
    where $(\sigma_1,\sigma_2)$ is a $1$-cover of $\sigma\subset[n].$

Again, the authors of \cite{ABBC21} proved functional versions of \eqref{eq:ineqSectABBC} and the authors of \cite{AAGJV} proved functional versions of \eqref{eq:AACJVs=1}. Sometimes, in order to derive these results, it turns out to be very helpful to use certain functional results of some geometric flavour. One of the most relevant tools in Convex Geometry is the Brunn concavity principle \cite{AGM15}, which states that the function
\[
f:H\rightarrow [0,+\infty),\quad\text{defined by}\quad f(x):=\vol{K\cap(x+H^\bot)},
\]
 is a $\frac{1}{i}$-concave function, where $K\in\mathcal K^n$ and $H\in\mathcal L^n_{n-i}$, for some $1\leq i\leq n-1$. This function belongs to the larger class of the so called log-concave functions. We say that $f:\mathbb R^n\rightarrow[0,+\infty)$ is a $t$-concave function, $t\geq 0$, if $f^{t}$ is a concave function. In particular, we say that $f$ is a log-concave (or $0$-concave) function if $\log(f)$ is concave. Equivalently, for every $x,y\in\mathbb R^n$ and $\lambda\in[0,1]$, it holds that $f((1-\lambda)x+\lambda y) \geq f(x)^{1-\lambda}f(y)^\lambda$. Moreover, let $\mathcal F(\mathbb R^n)$ be the set of integrable log-concave real functions whose domain is a convex subset of $\mathbb R^n$. 

Log-concave functions play a crucial role in Convex Geometry, as they can be seen as a proper generalization of the space of convex bodies (see \cite{KM05}) via, for instance, the natural identification of every $K\in\mathcal K^n$ with the characteristic function
\[
\chi_K:\mathbb R^n\rightarrow[0,\infty)\quad\text{where}\quad \chi_K(x):=\left\{\begin{array}{cc}
    1, & \text{ if }x\in K \\
    0, & \text{ otherwise.}
\end{array}\right.
\]
In the last years, many authors have proven results for log-concave functions extending some geometrical aspects, see for instance \cite{A19, ABG20, ABG202, AGJV16, AGJV18, AKM04, B88, BCF14, CF13, Co06, FSY22, FZ18, FM07, FM08, IN22, IT21, KM05, LMU24, LSW21, T25}.

Our first result is a functional inequality which involves $s$-covers of some $\sigma\subset[n]$ as well as a log-concave function $f$ and some restrictions of $f$ to certain linear subspaces. 
Motivated by the definition of reducible cover given in \cite{BT95}, we say that a $s$-cover $(\sigma_1,\dots,\sigma_m)$  of $\sigma\subseteq [n]$ is $1$-reducible if there exists a reordering of $(\sigma_1,\dots,\sigma_m)$ of the form $(\sigma_{1_1},\dots,\sigma_{1_{i_1}},\dots,\sigma_{s_1},\dots,\sigma_{s_{i_s}})$ such that $(\sigma_{j_1},\dots,\sigma_{j_{i_j}})$ is a $1$-cover of $\sigma$, for every $j=1,\dots,s$. 
Furthermore, given $(\sigma_1,\dots,\sigma_m)$ a $s$-cover of $\sigma\subseteq [n]$, we say that $(\overline{\sigma}_1,\dots,\overline{\sigma}_k)$ is the $1$-cover of $\sigma$ induced by the former $s$-cover, if $(\overline{\sigma}_1,\dots,\overline{\sigma}_k)$ are all the different non-empty possible subsets of the form $\cap_{j=1}^m \sigma_j^{\varepsilon(j)}$, where $\varepsilon(j)\in\{0,1\}$, $\sigma_j^0=\sigma_j$ and $\sigma_j^1=\sigma\setminus\sigma_j$ (see \cite{BT95, BKX23}). Finally, for every $X\subset\mathbb R^n$, we say that $\mathrm{conv}(X)$ is the convex hull of $X$, i.e.~the smallest convex set containing $X$.
\begin{thm}\label{thm:FunctionalRSscoverings}
    Let $f\in\mathcal F(\mathbb R^n)$ and let
    $(\sigma_1,\dots,\sigma_m)$  a $1$-reducible $s$-cover of $[n]$. Then,
    \begin{equation}\label{eq:FunctionalRSscoverings}
        \|f\|_\infty^{m-s}\left(\int_{\mathbb R^n}f(x)dx\right)^s\geq \frac{\prod_{j=1}^m|\sigma_j|!}{n!^s}\prod_{j=1}^m\int_{H_{\sigma_j}}f(x)dx.
    \end{equation}
    Moreover, equality holds above if and only if $f=\|f\|_\infty\chi_C$ where 
    \[
    C=\mathrm{conv}\left(\left\{\{0\}^{|\overline{\sigma}_1|}\times\ldots\times \{0\}^{|\overline{\sigma}_{j-1}|}\times K_j\times \{0\}^{|\overline{\sigma}_{j+1}|}\times\ldots\times \{0\}^{|\overline{\sigma}_k|}:j\in[k]\right\}\right),
    \]
    with $0\in K_j\in\mathcal{K}^{|\overline{\sigma}_j|}$ for $j=1,\dots,k$ and with $(\overline{\sigma}_1,\dots,\overline{\sigma}_k)$ being the $1$-cover of $[n]$ induced by the $s$-cover $(\sigma_1,\dots,\sigma_m)$.
\end{thm}

Notice that when $s=1$ and $m=2$ then Theorem \ref{thm:FunctionalRSscoverings} becomes the functional Rogers-Shephard type inequality \cite[Thm.~1.1]{AAGJV}. A remarkable consequence of Theorem \ref{thm:FunctionalRSscoverings} is the following result. It solves certain cases of the inequality formulated in \eqref{eq:ineqSectABBC}.

\begin{thm}\label{thm:sCoverDecom1Covers}
    Let $K\in\mathcal K^n$ and let
    $(\sigma_1,\dots,\sigma_m)$ a $s$-cover of $\sigma\subset[n]$ such that $(\sigma\setminus\sigma_1,\dots,\sigma\setminus\sigma_m)$ is a $1$-reducible $(m-s)$-cover of $\sigma$. Then,
    \[
   \vol{K}^{m-s} \max_{x\in\mathbb R^n}\vol{K\cap(x+H_\sigma^\bot}^{s}  \geq \frac{\prod_{j=1}^m(|\sigma|-|\sigma_j|)!}{|\sigma|!^{m-s}} \prod_{j=1}^m\vol{K\cap H_{\sigma_j}^\bot}.
    \]
    Moreover, equality holds above if and only if $0\in P_{H_\sigma}K=\mathrm{conv}(\{P_{H_{\sigma}}K\cap H_{\overline{\sigma}_i}:i=1,\dots,k\})$, where $(\overline{\sigma}_1,\dots,\overline{\sigma}_k)$ is the $1$-cover of $\sigma$ induced by the $s$-cover $(\sigma_1,\dots,\sigma_m)$, and where $K\cap(x+H_\sigma^\bot)$ are translates of each other for every $x\in P_{H_\sigma}K$.
\end{thm}

One can verify that on every optimal inequality above, we find certain unconditional convex bodies fulfilling them with equality. Recall that $K\in\mathcal K^n$ is said to be unconditional if we have that $(x_1,\dots,x_n)\in K$ if and only if $(\varepsilon_1 x_1,\dots,\varepsilon_n x_n)\in K$, for any choice of signs $\varepsilon_i\in\{-1,1\}$, $i=1,\dots,n$. For instance, the Hanner polytope
\[
K_\sigma:=\mathrm{conv}(\{\pm e_j:j\in\sigma\})\times \prod_{j\in[n]\setminus\sigma}\mathrm{conv}(\{\pm e_j\})
\]
fulfills Theorem \ref{thm:sCoverDecom1Covers} with equality for every $\sigma\subset[n]$ independently of the choice of the $s$-cover.  Inequality \eqref{eq:ineqSectABBC} was originally stated as
\begin{equation}\label{eq:IneqOriginalABBC}
    \vol{K}^s\max_{x\in\mathbb{R}^n}\vol{K\cap(x+ H_\sigma^\bot)}^{m-s}\geq\dfrac{\prod_{i=1}^m |\sigma_i|!}{(ns)!}\prod_{i=1}^m\vol{K\cap (H_{\sigma_i}\oplus H_\sigma^\bot)},
\end{equation}
for every $K\in\mathcal K^n$ and 
$(\sigma_1,\dots,\sigma_m)$ a $s$-cover of $\sigma\subset[n]$. Our next result improves \eqref{eq:IneqOriginalABBC} when assuming that $K$ is an unconditional set, and we prove it by using some ideas within the original paper of Meyer \cite{M88}.
\begin{thm}\label{thm:ImprovedAlonso_Meyer}
    Let $K\in\mathcal K^n$ be unconditional, let $\sigma\subset[n]$ and $\sigma_i:=\sigma\setminus\{i\}$, $i=1,\dots,|\sigma|$, be a $(|\sigma|-1)$-cover of $\sigma$. Then,
    \[
    \vol{K}^{|\sigma|-1} \vol{K\cap H_\sigma^\bot} \geq \frac{|\sigma|!}{n^{|\sigma|}} \prod_{j\in\sigma}\vol{K\cap(H_{\sigma_j}\oplus H_\sigma^\bot)}.
    \]
\end{thm}

The paper is organized as follows. In Section \ref{sec:ProofsOnIneqs} we will prove some functional inequalities that will allow us to prove the inequality of Theorem \ref{thm:FunctionalRSscoverings}. Then, using it, we will show also the inequality of Theorem \ref{thm:sCoverDecom1Covers}. Later in Section \ref{sec:ProofsOnEqualities}, we will prove all the characterizations of the equality cases of each inequality derived in Section \ref{sec:ProofsOnIneqs}. Later on in Section \ref{sec:FurtherBounds} we will derive Theorem \ref{thm:ImprovedAlonso_Meyer} as well as other inequalities that will allow us to provide for the first time a proper conjecture to the inequality \eqref{eq:ineqSectABBC}. 

\begin{conjecture}\label{conj:scover}
    Let $K\in\mathcal K^n$ and 
    $(\sigma_1,\dots,\sigma_m)$  a $s$-cover of $\sigma\subset[n].$ Then,
    \[
   \vol{K}^{m-s} \max_{x\in\mathbb R^n}\vol{K\cap(x+H_\sigma^\bot}^{s}  \geq \frac{\prod_{j=1}^m(|\sigma|-|\sigma_j|)!}{|\sigma|!^{m-s}} \prod_{j=1}^m\vol{K\cap H_{\sigma_j}^\bot}.
    \]
\end{conjecture}

Finally, we will prove in Section \ref{sec:Appendix} some estimates in order to quantify the improvement achieved in Theorem \ref{thm:ImprovedAlonso_Meyer} with respect to the previously known result in \eqref{eq:IneqOriginalABBC}.

\section{Proofs of the inequalities}\label{sec:ProofsOnIneqs}

We start this section by proving a generalization of \cite[Lemma 2.2]{AAGJV}, see also \cite[Lemma]{RS58}. Notice that it can be also derived from \eqref{eq:Liako19} by selecting the $1$-cover $(\sigma_1,\dots,\sigma_m)$ of $[i_1+\cdots+i_m]$ given by $\sigma_1=\{1,\dots,i_1\}$, $\sigma_{j+1}=\{1+\sum_{k=1}^{j}i_{k},\dots,\sum_{k=1}^{j+1}i_{k}\}$, for every $j=1,\dots,m-1$. However we write down the proof for the sake of completeness and also because we do not need to use such heavy machinery to show it as the geometric Brascamp-Lieb inequality used by Liakopoulos. For every $K\in\mathcal K^n$, we denote by $\mathrm{int}(K)$, $\mathrm{cl}(K)$, and $\mathrm{bd}(K)$ the interior, closure, and boundary of $K$, respectively.

\begin{lemma}\label{lem:generalizationRS}
    Let $K_j\in\mathcal K^{i_j}$ and $x_j\in\mathbb R^{i_j}$, $j=1,\dots,m$. Then
    \[
    \begin{split}
            \vol{\mathrm{conv}\left(\left\{  \{x_1\}\times\cdots\times\{x_{j-1}\}\times K_j\times\{x_{j+1}\}\times\cdots\times\{x_m\}:j=1,\dots,m\right\}\right)} & \\
           \geq \frac{\prod_{j=1}^m i_j!}{(i_1+\cdots+i_m)!} \prod_{j=1}^m\vol{K_j}. &
    \end{split}
    \]
    Moreover, equality holds above if and only if $x_j\in K_j$ for every $j=1,\dots,m$ or $\mathrm{int}(K_j)=\emptyset$ for some $j=1,\dots,m$.
\end{lemma}

\begin{proof}
    We prove it by induction on $m\geq 2$. The case $m=2$ is exactly Lemma 2.2 in \cite{AAGJV}. In order to show the inequality for an arbitrary $m>2$, observe that 
    \[
    \begin{split}
        & \vol{\mathrm{conv}\left(\left\{  \{x_1\}\times\cdots\times\{x_{j-1}\}\times K_j\times\{x_{j+1}\}\times\cdots\times\{x_m\}:j=1,\dots,m\right\}\right) } \\
        & = \vol{\mathrm{conv}(\{Q\times\{x_m\},(x_1,\dots,x_{m-1})\times K_m\})} \\
        & \geq \frac{(i_1+\cdots+i_{m-1})!i_m!}{(i_1+\cdots+i_m)!} \vol{Q}\vol{K_m},
    \end{split}
    \]
    where above we have applied the case $m=2$ to $K_m$ and $Q\in\mathcal K^{i_1+\cdots+i_{m-1}}$, with
    \[
    Q:=\mathrm{conv}\left(\left\{  \{x_1\}\times\cdots\times\{x_{j-1}\}\times K_j\times\{x_{j+1}\}\times\cdots\times\{x_{m-1}\}:j=1,\dots,m-1\right\}\right).
    \]
    Applying now the induction step to $\vol{Q}$ we thus get that
    \[
    \begin{split}
     & \frac{(i_1+\cdots+i_{m-1})!i_m!}{(i_1+\cdots+i_m)!} \vol{Q}\vol{K_m} \\
     & \geq \frac{(i_1+\cdots+i_{m-1})!i_m!}{(i_1+\cdots+i_m)!} \frac{\prod_{j=1}^{m-1}i_j!}{(i_1+\cdots+i_{m-1})!} \left(\prod_{j=1}^{m-1}\vol{K_j}\right)\vol{K_m} \\
     & = \frac{\prod_{j=1}^m i_j!}{(i_1+\cdots+i_m)!} \prod_{j=1}^m\vol{K_j},
     \end{split}
    \]
    as desired.

    We now show the case of equality. Equality would mean that there is equality on each inequality above. Notice that using the case of equality when $m=2$ we would have that either $(x_1,\dots,x_{m-1})\in Q$ and $x_m\in K_m$, or either $\mathrm{int}(Q)=\emptyset$ (i.e.~$\mathrm{int}(K_j)=\emptyset$ for some $j=1,\dots,m-1$) or $\mathrm{int}(K_m)=\emptyset$. Moreover, using the case of equality in the induction step, we would have that either $x_j\in K_j$, $j=1,\dots,m-1$, or $\mathrm{int}(K_j)=\emptyset$, for some $j=1,\dots,m-1$. Mixing all properties we obtain that either $\mathrm{int}(K_j)=\emptyset$, for some $j=1,\dots,m$, or $x_j\in K_j$, for every $j=1,\dots,m$, concluding the proof.
\end{proof}

We now need to introduce a convolution operator of log-concave functions that was already used in \cite{AAGJV} for two functions. For every $f_j\in\mathcal F(\mathbb R^n)$, $j=1,\dots,m$, let 
\[
\overline{\bigstar}_{j=1}^mf_j(z)=\sup \left\{\prod_{j=1}^mf_j(x_j)^{\lambda_j}: z=\sum_{j=1}^m\lambda_jx_j,\,\sum_{j=1}^m\lambda_j=1, \lambda_j\geq 0,\,j=1,\dots,m \right\}
\]
Notice that the operator $\overline{\bigstar}$ has the following properties.
\begin{propo}\label{prop:elementaryProp}
    Let $f,f_j\in\mathcal F(\mathbb R^n)$, $K_j\in\mathcal K^n$, $j=1,\dots,m$. Then
    \begin{enumerate}
        \item[$(i)$] $\mathrm{cl}\big(\mathrm{supp}(\overline{\bigstar}_{j=1}^mf_j)\big) = \mathrm{cl}\left(\mathrm{conv}(\{\mathrm{supp}(f_j):j=1,\dots,m\})\right)$.
        \item[$(ii)$] $\overline{\bigstar}_{j=1}^m\chi_{K_j}=\chi_Q$, where $Q=\mathrm{conv}(\{K_j:j=1,\dots,m\})$.
        \item[$(iii)$] If $f_j(x) \leq f(x)$, for every $x\in\mathbb R^n$ and every $j=1,\dots,m$, then
        \[
        \overline{\bigstar}_{j=1}^mf_j(x) \leq f(x), 
        \]
        for every $x\in\mathbb R^n$.
    \end{enumerate}
\end{propo}

\begin{proof}
    The proof of $(i)$ and $(ii)$ are obvious. For the proof of $(iii)$, we simply notice that if $x,x_j\in \mathbb R^n$, $\lambda_j\geq 0$, $j=1,\dots,m$, such that $\sum_{j=1}^m\lambda_j=1$, and $x=\sum_{j=1}^m\lambda_jx_j$, then the log-concavity of $f$ implies that
    \[
    \prod_{j=1}^mf_j(x_j)^{\lambda_j} \leq \prod_{j=1}^mf(x_j)^{\lambda_j} \leq f\left(\sum_{j=1}^m\lambda_jx_j\right)=f(x),
    \]
    and thus, taking the supremum over every such $x_j$ and $\lambda_j$, $j=1,\dots,m$, fulfilling the conditions described above, concludes $(iii)$.
\end{proof}

Next result is a functional inequality which provides a sharp lower bound for the integral of the convolution operator $\overline{\bigstar}$, and that will be crucial to show Theorem \ref{thm:FunctionalRSscoverings}. Let us also remember that from now on, we postpone the proofs of the characterization of the equality cases to Section \ref{sec:ProofsOnEqualities}.
\begin{lemma}\label{lem:OperatorLowerBound}
    Let $f_j\in\mathcal F(\mathbb R^{i_j})$, $j=1,\dots,m$, and let us define $F_j(x_1,\dots,x_m):=f_j(x_j)\chi_{\{0\}}(x_1,\dots,x_{j-1},x_{j+1},\dots,x_m)$, for every $x_j\in\mathbb R^{i_j}$, $j=1,\dots,m$. Then,
    \begin{equation}\label{eq:OperatorLowerBound}
        \max_{j=1,\dots,m} \|f_j\|_\infty^{m-1} \int_{\mathbb R^{i_1+\cdots+i_m}} \overline{\bigstar}_{j=1}^m F_j(x)dx 
        \geq \frac{\prod_{j=1}^m i_j!}{\left(\sum_{j=1}^m i_j\right)!} \prod_{j=1}^m\int_{\mathbb R^{i_j}}f_j(x_j)dx_j.  
    \end{equation}
    Moreover, equality holds above if and only if $f_j=\|f\|_\infty\chi_{K_j}$, where $0\in K_j\in\mathcal K^{i_j}$, $j=1,\dots,m$, and $\|f_1\|_\infty=\cdots=\|f_m\|_\infty$. 
\end{lemma}

\begin{proof}
    Let $z_j,x_j\in\mathbb R^{i_j}$ and $\lambda_j\geq 0$, $j=1,\dots,m$ be such that $z=(z_1,\dots,z_m)=\sum_{j=1}^m\lambda_jx_j$ with $\sum_{j=1}^m\lambda_j=1$. Then
    \[
    \prod_{j=1}^m F_j(x_j)^{\lambda_j} = \prod_{j=1}^m f_j\left(\frac{z_j}{\lambda_j}\right)^{\lambda_j} \geq \min\left\{f_j\left(\frac{z_j}{\lambda_j}\right):j=1,\dots,m\right\},
    \]
    and thus
    \begin{equation}\label{eq:Ineq1Lemma}
        \overline{\bigstar}_{j=1}^m F_j(z) \geq \sup_{\sum_{j=1}^m\lambda_j=1,\,\lambda_j\geq 0} \min\left\{f_j\left(\frac{z_j}{\lambda_j}\right):j=1,\dots,m\right\}.
        \end{equation}

    Let $A:=\|\overline{\bigstar}_{j=1}^m F_j\|_\infty=\max\{\|f_j\|_\infty:j=1,\dots,m\}$. Then,
    \[
    \begin{split}
       & \frac1A \int_{\mathbb R^{i_1+\cdots+i_m}} \overline{\bigstar}_{j=1}^m F_j(z)dz \\
       & = \int_0^1\vol{\left\{(z_1,\dots,z_m):\sup_{\sum_{j=1}^m\lambda_j=1}\prod_{j=1}^m f_j\left(\frac{z_j}{\lambda_j}\right)^{\lambda_j} \geq tA\right\} } dt \\
       & \underset{(i)}{\geq}  \int_0^1\vol{\left\{(z_1,\dots,z_m):\sup_{\sum_{j=1}^m\lambda_j=1}\min\left\{ f_j\left(\frac{z_j}{\lambda_j}\right):j=1,\dots,m \right\}\geq tA\right\} } dt \\
       & = \int_0^1 \vol{ \mathrm{conv}_{j=1,\dots,m}\left( \left\{ \{(0,\dots,0,x_j,0,\dots,0)\in\mathbb R^{i_1+\cdots+i_m}:f_j(x_j) \geq tA\} \right\}\right) } dt \\
       & \underset{(ii)}{\geq} \frac{\prod_{j=1}^m i_j!}{(i_1+\cdots+i_m)!} \int_0^1 \prod_{j=1}^m \vol{ \{(0,\dots,0,x_j,0,\dots,0)\in\mathbb R^{i_1+\cdots+i_m}:f_j(x_j) \geq tA\}  } dt,
    \end{split}
    \]
    where inequality $(i)$ follows from \eqref{eq:Ineq1Lemma} and inequality $(ii)$ follows from Lemma \ref{lem:generalizationRS}. Evidently, the last line above can be rewritten as
    \begin{equation}\label{eq:Ineq2Lemma}
    \frac{\prod_{j=1}^m i_j!}{(i_1+\cdots+i_m)!} \int_0^1 \prod_{j=1}^m \left( \int_{\mathbb R^{i_j}}\chi_{\{f_j(x_j)\geq tA\}}(x_j)dx_j \right)dt.
    \end{equation}
    Let us notice that the term within the integral above equals to
    \[
    \begin{split}
        & \int_0^1\left(  \int_{\mathbb R^{i_1}\times \cdots\times\mathbb R^{i_m}} \chi_{\{f_1(x_1)\geq tA\}}(x_1)\cdots\chi_{\{f_m(x_m)\geq tA\}}(x_m)dx_1\cdots dx_m \right)dt  \\
        & = \int_0^1\left(  \int_{\mathbb R^{i_1}\times \cdots\times\mathbb R^{i_m}} \chi_{\{\min\{f_1(x_1),\dots,f_m(x_m)\}\geq tA\}}(x_1,\dots,x_m) dx_1\cdots dx_m \right)dt \\
        & = \int_{\mathbb R^{i_1}\times \cdots\times\mathbb R^{i_m}} \left( \int_0^1 \chi_{\{\min\{f_1(x_1),\dots,f_m(x_m)\}\geq tA\}}(t) dt\right) dx_1\cdots dx_m \\
        & = \int_{\mathbb R^{i_1}\times \cdots\times\mathbb R^{i_m}} \min\left\{\frac{f_1(x_1)}{A},\dots,\frac{f_m(x_m)}{A}\right\} dx_1\cdots dx_m,
    \end{split}
    \]
    where above we have changed the order of the integrals by Fubini's theorem. Therefore \eqref{eq:Ineq2Lemma} turns out to be
    \begin{equation*}
    \begin{split}
    & \frac{\prod_{j=1}^m i_j!}{(i_1+\cdots+i_m)!} \int_{\mathbb R^{i_1}\times \cdots\times\mathbb R^{i_m}} \min\left\{\frac{f_j(x_j)}{A}:j=1,\dots,m\right\} dx_1\cdots dx_m \\
    & \underset{(iii)}{\geq} \frac{\prod_{j=1}^m i_j!}{(i_1+\cdots+i_m)!} \prod_{j=1}^m\int_{\mathbb R^{i_j}} \frac{f_j(x_j)}{A}dx_j,
    \end{split}
    \end{equation*}
    where inequality $(iii)$ follows from the fact that $\frac{f_j(x_j)}{A}\leq 1$ for every $j=1,\ldots,m$. This 
    concludes the proof of the inequality.\end{proof}
    
We are now able to show the inequality of Theorem \ref{thm:FunctionalRSscoverings}. 
\begin{proof}[Proof of Theorem \ref{thm:FunctionalRSscoverings}]
    Since the $s$-cover $(\sigma_1,\dots,\sigma_m)$ of $[n]$ is $1$-reducible, let \\$(\sigma_{1_1},\dots,\sigma_{1_{i_1}},\dots,\sigma_{s_1},\dots,\sigma_{s_{i_s}})$ be a reordering of the former cover such that\\ 
    $(\sigma_{j_1},\dots,\sigma_{j_{i_j}})$ is a $1$-cover of $[n]$, for every $j=1,\dots,s$. Let us consider the functions
    \[
    F_{j_\ell}(x_{j_1},\dots,x_{j_{i_j}}) := f(x_{j_\ell}) \chi_{\{0\}} (x_{j_1},\dots,x_{j_{\ell-1}},x_{j_{\ell+1}},\dots,x_{j_{i_j}}),
    \]
    where $x_{j_\ell}\in\mathbb R^{|\sigma_{j_\ell}|}$, for every $j=1,\dots,s$ and every $\ell=1,\dots,i_j$. Since 
    \[
    \max\left\{\|f|_{H_{\sigma_{j_\ell}}}\|_\infty:\ell=1,\dots,i_j\right\} \leq \|f\|_\infty
    \]
    for every $j=1,\dots,s$, by $(iii)$ of Proposition \ref{prop:elementaryProp} and Lemma \ref{lem:generalizationRS} we then get that 
    \[
    \begin{split}
        \|f\|_\infty^{m-s} \left(\int_{\mathbb R^n}f(z)dz\right)^s & = \prod_{j=1}^s \left( \|f\|_\infty^{i_j-1} \int_{\mathbb R^n} f(z)dz\right) \\
        & \geq \prod_{j=1}^s\left(  \max_{\ell=1,\dots,i_j}\{ \|f|_{H_{\sigma_{j_\ell}}}\|_\infty\}^{i_j-1} \int_{\mathbb R^n} \overline{\bigstar}_{\ell=1}^{i_j}F_{j_\ell}(z)dz \right) \\
        & \geq \prod_{j=1}^s \left( \frac{\prod_{\ell=1}^{i_j}|\sigma_{j_\ell}|!}{(\sum_{\ell=1}^{i_j}|\sigma_{j_\ell}|)!} \prod_{\ell=1}^{i_j} \int_{H_{\sigma_{j_\ell}}} f(x)dx \right) \\
        & = \frac{\prod_{j=1}^m|\sigma_j|!}{n!^s} \prod_{j=1}^m \int_{H_{\sigma_j}} f(x)dx,
    \end{split}
    \]
    where in the last equality we have also used the fact that $\sum_{\ell=1}^{i_j}|\sigma_{j_\ell}|=n$, for every $1$-cover $(\sigma_{j_1},\dots,\sigma_{j_{i_j}})$ of $[n]$, for every $j=1,\dots,s$. Thus, we conclude the proof of the inequality.
\end{proof}

We are now able to derive the inequality of Theorem \ref{thm:sCoverDecom1Covers}. 
\begin{proof}[Proof of Theorem \ref{thm:sCoverDecom1Covers}]
    Since $(\sigma_1,\dots,\sigma_m)$ is a $s$-cover of $\sigma$, then $(\sigma\setminus\sigma_1,\dots,\sigma\setminus\sigma_m)$ is a $(m-s)$-cover of $\sigma$. Let us consider
    \[
    f:H_\sigma\rightarrow[0,+\infty)\quad\text{were}\quad f(x)=\vol{K\cap(x+H_\sigma^\bot)}.
    \]
    By the Brunn Concavity Principle, $f$ is a $\frac{1}{n-|\sigma|}$-concave function, and thus, $f$ is log-concave too. Since $K$ is compact, we actually have that $f\in\mathcal F(H_\sigma)$. Now applying Theorem \ref{thm:FunctionalRSscoverings} to $f$ and the $1$-reducible $(m-s)$-cover $(\sigma\setminus\sigma_1,\dots,\sigma\setminus\sigma_m)$ of $\sigma$, we get that
    \[
    \begin{split}
        \max_{x\in H_\sigma}\vol{K\cap(x+H_\sigma^\bot)}^s \vol{K}^{m-s} & = \|f\|_\infty^s \left(\int_{H_\sigma}f(x)dx\right)^{m-s} \\
        & \geq \frac{\prod_{j=1}^m(|\sigma|-|\sigma_j|)!}{|\sigma|!^{m-s}} \prod_{j=1}^m\int_{H_{\sigma\setminus\sigma_j}} f(x)dx \\
        & = \frac{\prod_{j=1}^m(|\sigma|-|\sigma_j|)!}{|\sigma|!^{m-s}} \prod_{j=1}^m \vol{K\cap H_{\sigma_j}^\bot},
    \end{split}
    \]
    which concludes the proof of the inequality.
\end{proof}

As a consecuence of Theorem \ref{thm:FunctionalRSscoverings} we can prove a functional version of Theorem \ref{thm:sCoverDecom1Covers} for log-concave functions. In order to prove it, we are using an important result, the  Prékopa-Leindler inequality (see \cite{Le72,Pr71}), which states that for $f,\, g,\, h:\mathbb{R}^n\to[0,\infty)$ integrable functions such that $h(\lambda x+(1-\lambda)y)\geq f(x)^\lambda g(y)^{1-\lambda}$ for every $x,y\in\mathbb{R}^n$ and some $\lambda\in[0,1]$ we have that
\begin{equation}\label{eq:PrekLeind}
\int_{\mathbb{R}^n}h(x)dx\geq\left(\int_{\mathbb{R}^n}f(x)dx\right)^\lambda\left(\int_{\mathbb{R}^n}g(x)dx\right)^{1-\lambda}.
\end{equation}
  
\begin{cor}\label{cor:funcNewIneq}
Let $f\in\mathcal F(\mathbb R^n)$ and let
$(\sigma\setminus \sigma_1,\ldots,\sigma\setminus\sigma_m)$ a $1$-reducible $(m-s)$-cover of $\sigma\subset [n]$. Then,
\[
\left(\max_{x\in\mathbb{R}^n}\int_{x+H_\sigma^\bot}f(y)dy\right)^{s}\left(\int_{\mathbb{R}^n}f(x)dx\right)^{m-s}\geq \frac{\prod_{j=1}^m(|\sigma|-|\sigma_j|)!}{|\sigma|!^{m-s}} \prod_{j=1}^m\int_{H_{\sigma_j}^\bot}f(x)dx.
\]
\end{cor}
\begin{proof}
    Let us consider the function $F:H_\sigma\to[0,\infty)$ given by 
    \[
    F(x)=\int_{x+H_\sigma^\bot}f(y)dy
    \]
    for any $x\in H_\sigma$. Using the log-concavity of the function $f$ we obtain that 
    \[
    \begin{split}
      &f(\lambda x+(1-\lambda)y+t)\chi_{\lambda x+(1-\lambda)y+H_\sigma^\bot}(\lambda x+(1-\lambda)y+t)\\
      &\geq \left(f(x+t)\chi_{x+H_\sigma\bot}(x+t)\right)^\lambda\left(f(x+t)\chi_{y+H_\sigma^\bot}(y+t)\right)^{1-\lambda},
    \end{split}
    \]
    and using Prékopa-Leindler inequality \eqref{eq:PrekLeind} then
    \[
    \begin{split}
        F(\lambda x+(1-\lambda)y)&=\int_{\mathbb{R}^n}f(z)\chi_{\lambda x+(1-\lambda)y+H_\sigma^\bot}(z)dz\\
        &\geq\left(\int_{\mathbb{R}^n}f(z)\chi_{x+H_\sigma^\bot}(z)\right)^\lambda\left(\int_{\mathbb{R}^n}f(z)\chi_{y+H_\sigma^\bot}(z)\right)^{1-\lambda}\\
        &=F(x)^\lambda F(y)^{1-\lambda},
    \end{split}
    \]
    i.e.~$F$ is a log-concave function too.
    
    Finally, applying Theorem \ref{thm:FunctionalRSscoverings} to $F$ and the $1$-reducible $(m-s)$-cover $(\sigma\setminus\sigma_1,\ldots,\sigma\setminus\sigma_m) $ of $\sigma$ we get that
    \[
    \begin{split}
        \left(\max_{x\in\mathbb{R}^n}\int_{x+H_\sigma^\bot}f(y)dy\right)^{s}\left(\int_{\mathbb{R}^n}f(x)dx\right)^{m-s}&=\|F\|_\infty^s\left(\int_{H_\sigma}F(x)dx\right)^{m-s}\\
        &\geq \frac{\prod_{j=1}^m(|\sigma|-|\sigma_j|)!}{|\sigma|!^{m-s}} \prod_{j=1}^m\int_{H_{\sigma\setminus\sigma_j}} F(x)dx\\
        &=\frac{\prod_{j=1}^m(|\sigma|-|\sigma_j|)!}{|\sigma|!^{m-s}} \prod_{j=1}^m\int_{H_{\sigma_j}^\bot} f(x)dx,
    \end{split}
    \]
    which concludes the proof.
\end{proof}
 Notice that if we take $f=\chi_K$, for some convex body $K$, then Corollary \ref{cor:funcNewIneq} recovers the geometric inequality of Theorem \ref{thm:sCoverDecom1Covers}.

\section{Proofs of the equality cases}\label{sec:ProofsOnEqualities}

We start this section by proving the following lemma that generalizes Lemma 2.3 in \cite{AAGJV} to arbitrary number of functions. We prove it in this section because the novelty here is specially in the equality case. Let us also mention that in every equality case, the values that every log-concave function $f$ takes on $\mathrm{bd}(\mathrm{supp}(f))$ do not change any result where what we compute is the integral of the function, as long as we preserve the log-concavity of $f$. This is why sometimes we just always replace (even without saying it explicitly) $f$ by the superior limit function (see \cite[Lemma~2.1]{Co06}
\[
\hat{f}(x):=\limsup_{y\rightarrow x,\,y\in\mathrm{int}(\mathrm{supp}(f))}f(y),
\]
defined in $\mathrm{supp}(\hat{f}):=\mathrm{cl}(\mathrm{supp}(f))$. In particular, if $f\in\mathcal F(\mathbb R^n)$ then $\hat{f}\in\mathcal F(\mathbb R^n)$ with $f(x)\leq \hat{f}(x)$, for every $x\in\mathbb R^n$, and $\int_{\mathbb R^n} f(x)dx=\int_{\mathbb R^n} \hat{f}(x)dx$.
\begin{lemma}\label{lem:MinProdgeneralization}
    Let $f_j\in\mathcal F(\mathbb R^{i_j})$ be such that $f_j(x_j)\leq 1$ for every $x_j\in\mathbb R^{i_j}$ and $j=1,\dots,m$. Then
    \[
    \int_{\mathbb R^{i_1+\cdots+i_m}} \min\left\{f_j(x_j):j=1,\dots,m\right\} d(x_1,\dots,x_m) \geq \prod_{j=1}^m\int_{\mathbb R^{i_j}}f_j(x_j)dx_j.
    \]
    Moreover, equality holds above if and only if there exists $\sigma\subset[m]$ with $|\sigma|=m-1$, and some $K_j\in\mathcal K^{i_j}$ with $j\in\sigma$, such that $f_j(x_j)=\chi_{K_j}$, for every $j\in\sigma$.
\end{lemma}

\begin{proof}
We prove the result by induction on $m$. We know that the result is true for $m=2$ by Lemma 2.3 in \cite{AAGJV}. For any $m>2$ we have that
\[
\begin{split}
    & \int_{\mathbb R^{i_1+\cdots+i_m}} \min\left\{f_j(x_j):j=1,\dots,m\right\} d(x_1,\dots,x_m) \\
    & = \int_{\mathbb R^{i_1}\times\mathbb R^{i_2+\cdots+i_m}} \min\left\{f_j(x_1), \min\left\{  f_j(x_j):j=2,\dots,m\right\} \right\} d(x_1,\dots,x_m) \\
    & \geq \int_{\mathbb R^{i_1}} f_1(x_1)dx_1 \int_{\mathbb R^{i_2+\cdots+i_m}} \min\left\{  f_j(x_j):j=2,\dots,m\right\} d(x_2,\dots,x_m) \\
    & \geq \int_{\mathbb R^{i_1}} f_1(x_1)dx_1 \prod_{j=2}^m \int_{\mathbb R^{i_j}} f_j(x_j)dx_j,
\end{split}
\]
which shows the inequality.

If we have equality, then we would have equality on each inequality sign above. On the one hand, the case of $m=2$ says that equality holds if and only if either $f_1(x_1)=\chi_{K_1}$ or $\min\{f_j(x_j):j=2,\dots,m\}=\chi_{C}(x_2,\dots,x_m)$, for some $K_1\in\mathcal K^{i_1}$ or some $C\in\mathcal K^{i_2+\cdots+i_m}$. On the other hand, equality holds by the induction step if and only if there exist $\sigma'\subset\{2,\dots,m\}$ with $|\sigma'|=m-2$ (and without loss of generality, say that $\sigma'=\{2,\dots,m-1\}$) and some $K_j\in\mathcal K^{i_j}$, $j=2,\dots,m-1$, such that $f_j(x_j)=\chi_{K_j}(x_j)$, for every $j=2,\dots,m-1$.

Thus, we get two possible cases. In the first one, $f_j(x_j)=\chi_{K_j}(x_j)$, $j=1,\dots,m-1$, and the characterization holds. In the second one, $f_j(x_j)=\chi_{K_j}(x_j)$, $j=2,\dots,m-1$, and $\min\{f_2(x_2),\dots,f_m(x_m)\}=\chi_C(x_2,\dots,x_m)$. Hence
\[
\begin{split}
\chi_C(x_2,\dots,x_m) & = \min\{f_2(x_2),\dots,f_m(x_m)\} \\
& = \min\{\chi_{K_2}(x_2),\dots,\chi_{K_{m-1}}(x_{m-1}),f_m(x_m)\} \\
& =f_m(x_m)
\end{split}
\]
for every $x_i\in\mathcal K_i$, $i=2,\dots,m-1$ and $x_m\in\mathrm{supp}(f_m)$. Hence, $C=K_2\times\cdots\times K_{m-1}\times\mathrm{supp}(f_m)$, and therefore $f_m(x_m)=\chi_{K_m}(x_m)$, where $K_m=\mathrm{cl}(\mathrm{supp}(f_m))$. This concludes the proof.
\end{proof}

We are now able to prove the characterization of the equality in Lemma \ref{lem:OperatorLowerBound}.
\begin{proof}[Equality case in Lemma \ref{lem:OperatorLowerBound}]
    Equality in \eqref{eq:OperatorLowerBound} means that there is equality in $(i)$, $(ii)$, and $(iii)$ in the proof of Lemma \ref{lem:OperatorLowerBound}.

    Equality in $(iii)$ holds by Lemma \ref{lem:MinProdgeneralization} if and only if (after reordering) there exist $K_j\in\mathcal K^{i_j}$, $j=1,\dots,m-1$, such that $\frac{f_j}{A}=\chi_{K_j}$ for every $j=1,\dots,m-1$. In particular, $\|f_j\|_\infty=A=\|\overline{\bigstar}_{j=1}^m F_j\|_\infty=\max\{\|f_j\|_\infty:j=1,\dots,m\}$ for every $j=1,\dots,m-1$. Thus, it remains to show that $f_m$ is a multiple of a characteristic function too.

    Equality in $(ii)$ implies, by Lemma \ref{lem:generalizationRS}, that 
    \[
    0\in\{x\in\mathbb R^{i_j}:f_j(x) \geq tA\}
    \]
    for every $j=1,\dots,m$ and every $t\in(0,\frac{\|f_m\|_\infty}{A}]$. Notice that if $t>\frac{\|f_m\|_\infty}{A}$ then $\{x\in\mathbb R^{i_m}:f_m(x)\geq tA\}=\emptyset$. In particular, taking $t_0:=\frac{\|f_m\|_\infty}{A}$ then $0\in \{x\in\mathbb R^{i_m}:f_m(x) \geq \|f_m\|_\infty\}$, i.e., $f_m(0)=\|f\|_\infty$.

    Furthermore, since $f_j=A\chi_{K_j}$ for $j=1,\dots,m-1$, letting $t\in(0,\frac{\|f_m\|_\infty}{A}]$, we have that
    \[
    \begin{split}
        &\{(0,\dots,0,x_j,0,\dots,0):f_j(x_j)\geq tA\} =\{(0,\dots,0,x_j,0,\dots,0):\chi_{K_j}(x_j) \geq t\} \\
        &=\{0\}^{i_1}\times\cdots\times\{0\}^{i_{j-1}}\times K_j\times\{0\}^{i_{j+1}}\times\cdots\times\{0\}^{i_m}
    \end{split}
    \]
    where $(0,\dots,0,x_j,0,\dots,0)\in\mathbb R^{i_1}\times\cdots\times\mathbb R^{i_{j-1}}\times \mathbb R^{i_{j}}\times\mathbb R^{i_{j+1}}\times\cdots\times\mathbb R^{i_m}$. Thus
    \[
    0\in K_1\times\cdots\times K_{m-1}\times\{x\in\mathbb R^{i_m}:f_m(x) \geq tA\},
    \]
    for every $t\in(0,\frac{\|f_m\|_\infty}{A}]$.

    Let us assume now that there is equality in $(i)$. First, we show that $\|f_m\|_\infty=A$ by contradiction. Hence, let us assume that $\|f_m\|_\infty<A$. Equality in $(i)$ means that, for every $t\in(0,1)$, equality holds between the sets
    \begin{equation}\label{eq:EqualConv}
    \mathrm{conv}\left((K_1\times\cdots\times K_{m-1}\times\{0\}^{i_m}) \cup \{(0,\dots,0,x_m):f_m(x_m)\geq tA\}\right)
    \end{equation}
    and
    \begin{equation}\label{eq:EqualInSupProd}
    \begin{split}
     \left\{(z_1,\dots,z_m):\sup_{\sum_{j=1}^m\lambda_j=1}\prod_{j=1}^m f_j\left(\frac{z_j}{\lambda_j}\right)^{\lambda_j} \geq tA\right\} & \\
    = \left\{(z_1,\dots,z_m):\overline{\bigstar}_{j=1}^m F_j(z) \geq t \|\overline{\bigstar}_{j=1}^m F_j\|_\infty\right\}. &
    \end{split}
    \end{equation}
    In particular, equality holds for every $t\in(\frac{\|f_m\|_\infty}{A},1)$. However, first let us notice that the set in \eqref{eq:EqualConv} has empty interior, since
    \[
    \{x_m\in\mathbb R^{i_m}:f_m(x_m) \geq tA\} = \emptyset
    \]
    for every $t\in(\frac{\|f_m\|_\infty}{A},1)$, and hence the set described in \eqref{eq:EqualConv} has volume equal to $0$. Second, let us prove that the volume of the set in \eqref{eq:EqualInSupProd} has strictly positive volume. Taking into account that $f_j=A\chi_{K_j}$ for every $j=1,\dots,m-1$, let $z_j\in\lambda_jK_j\subset K_j$ (this last inclusion is true due to $0\in K_j$) for certain $\lambda_j\in(0,1)$, and $j=1,\dots,m-1$, since
    \[
    \sup_{\sum_{j=1}^m\lambda_j=1}\prod_{j=1}^m f_j\left(\frac{z_j}{\lambda_j}\right)^{\lambda_j} \geq tA 
    \quad\Leftrightarrow\quad 
    \sup_{\sum_{j=1}^m\lambda_j=1}\prod_{j=1}^m \left(\frac{f_j\left(\frac{z_j}{\lambda_j}\right)}{A}\right)^{\lambda_j} \geq t,
    \]
    the term inside the supremum in the set within \eqref{eq:EqualInSupProd} can be rewritten as
    \begin{equation}\label{eq:LowerInSupProd}
        \left(\frac{f_m\left(\frac{z_m}{\lambda_m}\right)}{A}\right)^{\lambda_m},
    \end{equation}
    with $\lambda_1+\cdots+\lambda_m=1$, for some $\lambda_m \in(0,1)$. Moreover, since $f_m(0)=\|f_m\|_\infty$, given $\mu>0$ there exists a closed ball containing the origin $B_\mu\subset\mathbb R^{i_m}$ fulfilling $f_m(z_m) \geq \|f_m\|_\infty-\mu$ for every $z_m\in B_\mu$. Equivalently, if $\frac{z_m}{\lambda_m}\in B_\mu$ then, since $0\in B_\mu$, $z_m\in\lambda_m B_\mu \subset B_\mu$ and thus $f_m(\frac{z_m}{\lambda_m}) \geq \|f_m\|_\infty-\mu$. Therefore, we could bound from below \eqref{eq:LowerInSupProd} by
    \[
    \left(\frac{\|f_m\|_\infty-\mu}{A}\right)^{\lambda_m}.
    \]
    In particular, let $\lambda_{m}(t)\in(0,1)$ be small enough such that
    \[
    \left(\frac{f_m\left(\frac{z_m}{\lambda_{m}(t)}\right)}{A}\right)^{\lambda_{m}(t)} \geq \left(\frac{\|f_m\|_\infty-\mu}{A}\right)^{\lambda_{m}(t)} \geq t
    \]
    for every $t\in(\frac{\|f_m\|_\infty}{A},1)$. Therefore, taking $\lambda_j:=\frac{1-\lambda_{m}(t)}{m-1}$ for $j=1,\dots,m-1$ we have that
    \[
    \sup_{\sum_{j=1}^m\lambda_j=1} \left\{ \left( \frac{f_m\left( \frac{z_m}{\lambda_m} \right)}{A} \right)^{\lambda_m} \right\}
    \geq \left(\frac{f_m\left(\frac{z_m}{\lambda_{m}(t)}\right)}{A}\right)^{\lambda_{m}(t)} \geq t,
    \]
    whenever $z_j\in \frac{1-\lambda_{m}(t)}{m-1} K_j$ for every $j=1,\dots,m-1$ and $z_m\in\lambda_{m}(t)B_\mu$. Thus, for every $t\in(\frac{\|f_m\|_\infty}{A},1)$ we find that the set
    \[
    (\lambda_1K_1)\times\cdots\times(\lambda_{m-1}K_{m-1})\times(\lambda_{m}(t) B_\mu) 
    \]
    is contained in the set in \eqref{eq:EqualInSupProd}, thus proving that it has strictly positive volume, a contradiction. Hence, $\|f_m\|_\infty=A$.

    We finally show that $\frac{f_m}{A}$ is a characteristic function. Let us suppose that there exists $x_0\in\mathrm{int}(\mathrm{supp}(f))$ such that
    \[
    a:=\frac{f_m(x_0)}{A} \in (0,1).
    \]
    Let $\varepsilon_0>0$ be such that $\varepsilon_0<\min\{a,1-a\}$. Since $f_m$ is a log-concave function, then it is continuous in $\mathrm{int}(\mathrm{supp}(f))$, and hence for every $\varepsilon\in(0,\varepsilon_0)$ there exists an Euclidean ball $U_\varepsilon$ centred on $x_0$ such that
    \[
    a-\varepsilon<\frac{f_m(x)}{A} < a+\varepsilon,
    \]
    for every $x\in U_\varepsilon$. It holds that
    \[
    U_\varepsilon \cap \{x\in\mathbb R^{i_m}: f_m(x)\geq (a+\varepsilon)A\} = \emptyset
    \]
    and 
    \[
    U_\varepsilon \subset \{x\in\mathbb R^{i_m}:f_m(x)\geq (a-\varepsilon)A\}.
    \]
    Since for every $\varepsilon\in(0,\varepsilon_0)$ we have that
    \[
    \lim_{\theta\rightarrow 0}(a-\varepsilon)^\theta=1,
    \]
    there exists $\theta_0(\varepsilon)\in(0,1)$ such that $a+\varepsilon < (a-\varepsilon)^\theta <1$ for every $\theta\in(0,\theta_0(\varepsilon))$. Therefore, for every $x\in U_\varepsilon$ and $\theta\in(0,\theta_0(\varepsilon))$ we have that
    \[
    \left(\frac{f_m(x)}{A}\right)^\theta > (a-\varepsilon)^\theta >a+\varepsilon
    \]
    and thus the set
    \[
    \begin{split}
    \mathcal P  := & \mathrm{conv}\left[ \{K_1\times\cdots\times K_{m-1}\times \{0\}^{i_m},\right. \\ 
    &  \left. \{0\}^{i_1}\times\cdots\times\{0\}^{i_{m-1}}\times(U_\varepsilon\cup\{x:f_m(x)\geq(a+\varepsilon)A\})\}\right]    
    \end{split}
    \]
    is contained within
    \[
    \mathcal M_1:=\left\{ (z_1,\dots,z_m): \sup_{\sum_{j=1}^m\lambda_j} \left(\frac{f_m\left(\frac{z_m}{\lambda_m}\right)}{A}\right)^{\lambda_m} \prod_{j=1}^{m-1}\left(\frac{\chi_{K_j}\left(\frac{z_j}{\lambda_j}\right)}{A}\right)^{\lambda_j} \geq a+\varepsilon \right\}.
    \]
    Furthermore, since $U_\varepsilon \cap \{x\in\mathbb R^{i_m}:f_m(x)\geq(a+\varepsilon)A\} = \emptyset$ and both sets are convex, there exists a hyperplane separating both sets. Thus, the volume of $\mathcal P$ is strictly larger than that of
    \[
    \begin{split}
   \mathcal M_2:= & \mathrm{conv}\left[ \Big\{K_1\times\cdots\times K_{m-1}\times \{0\}^{i_m},\right. \\ 
    &  \left. \{0\}^{i_1}\times\cdots\times\{0\}^{i_{m-1}}\times\{x:f_m(x)\geq(a+\varepsilon)A\}\Big\}\right]    
    \end{split}
    \]
    
    Hence, we have obtained that $\vol{\mathcal M_2} < \vol{\mathcal P} \leq\vol{\mathcal M_1}$ for every $\varepsilon\in(0,\varepsilon_0)$, contradicting the equality in $(i)$, since we get that $(i)$ is strict for every $t\in(a,a+\varepsilon_0)$, and thus, the corresponding integrals would be strict too. 
    Therefore, $\frac{f_m}{A}$ takes just values $0$ or $1$ within $\mathrm{int}(\mathrm{supp}(f))$, i.e.~$\frac{f_m}{A}$ is a characteristic function too, concluding the proof of the equality case.
\end{proof}

We can now address the characterization of the equality case of Theorem \ref{thm:FunctionalRSscoverings}. We first show an elementary lemma.
\begin{lemma}\label{lem:CapWithSubspace}
    Let $K_j\in\mathcal K^{i_j}$ for $j=1,\dots,m$, with $m\geq 2$ and let $\sigma\subset[i_1+\cdots+i_m]$. Then
    \[
    \begin{split}
         \mathrm{conv}\left( \{ \{0\}^{i_1}\times\cdots\times\{0\}^{i_{j-1}}\times K_j \times \{0\}^{i_{j+1}}\times\cdots\times\{0\}^{i_m}:j=1,\dots,m \} \right) \cap H_\sigma = & \\
         \mathrm{conv}\left( \{ \{0\}^{i_1}\times\cdots\times\{0\}^{i_{j-1}}\times (K_j \cap H_\sigma) \times \{0\}^{i_{j+1}}\times\cdots\times\{0\}^{i_m}:j=1,\dots,m \} \right) &
    \end{split}
    \]
\end{lemma}

\begin{proof}
    The inclusion $\supset$ is clear. Let us therefore show the inclusion $\subset$. We prove it by induction on $m\geq 2$. Let first $m=2$, and let 
    \[
    x\in\mathrm{conv}\left(\{K_1\times\{0\}^{i_2},\{0\}^{i_2}\times K_2\}\right) \cap H_\sigma.
    \]
    Then, there exist $\lambda\in[0,1]$, $x_1\in K_1\times\{0\}^{i_2}$, and $x_2\in \{0\}^{i_2}\times K_2$ such that $x=\lambda x_1+(1-\lambda)x_2$. Thus $x_1=(x_{1_1},\dots,x_{1_{i_1}},0,\dots,0)$ and $x_2=(0,\dots,0,x_{2_{i_1+1}},\dots,x_{2_{i_1+i_2}})$, and hence we have that
    \[
    x=\left( \lambda x_{1_1},\dots, \lambda x_{1_{i_1}}, (1-\lambda)x_{2_{i_1+1}},\dots,(1-\lambda)x_{2_{i_1+i_2}} \right).
    \]
    Moreover, we notice that
    $$
    \langle x,e_\ell\rangle = \left\{
    \begin{array}{ll}
        \lambda x_{1_\ell} &  \text{ if } \ell\in[i_1] \\
        (1-\lambda)x_{2_{\ell}} & \text{ if } \ell\in([i_1+i_2]\setminus[i_1])  
    \end{array}
    \right.$$
    and since $x\in H_\sigma$, we have that $\langle x,e_\ell\rangle =0$ if $\ell\notin \sigma$. 
    Therefore, $x_1,x_2\in H_\sigma$ and the inclusion is proved.

    We now assume that the result is true for at most $m-1$ subsets. Then applying the case of $m=2$ and the induction step we get that
    \[
    \begin{split}
       &  \mathrm{conv}\left( \{ \{0\}^{i_1}\times\cdots\times\{0\}^{i_{j-1}}\times K_j \times \{0\}^{i_{j+1}}\times\cdots\times\{0\}^{i_m}:j=1,\dots,m \} \right) \cap H_\sigma \\
       & = \mathrm{conv}( \{\mathrm{conv}\left( \{ \{0\}^{i_1}\times\cdots\times\{0\}^{i_{j-1}}\times K_j \times \{0\}^{i_{j+1}}\times\cdots\times\{0\}^{i_m}:j\in[m-1] \} \right), \\
       & (\{0\}^{i_1}\times\cdots\times\{0\}^{i_{m-1}}\times K_m)\} \cap H_\sigma \\
       & = \mathrm{conv}( \{(\{0\}^{i_1}\times\cdots\times\{0\}^{i_{m-1}}\times (K_m\cap H_\sigma)), \\
       & \mathrm{conv}\left( \{ \{0\}^{i_1}\times\cdots\times\{0\}^{i_{j-1}}\times K_j \times \{0\}^{i_{j+1}}\times\cdots\times\{0\}^{i_m}:j\in[m-1] \} \right) \cap H_\sigma\} \\
       & = \mathrm{conv}(\{\{0\}^{i_1}\times\cdots\times\{0\}^{i_{j-1}}\times (K_j\cap H_\sigma) \times\{0\}^{i_{j+1}}\times\cdots\times\{0\}^{i_m}:j=1,\dots,m\}),
    \end{split}
    \]
    concluding the proof.
\end{proof}

We are finally able to show the characterization of the equality case of Theorem \ref{thm:FunctionalRSscoverings}.
\begin{proof}[Proof of the equality case of Theorem \ref{thm:FunctionalRSscoverings}]
    Let us suppose we have equality in \eqref{eq:FunctionalRSscoverings}. It is then clear that for every $j=1,\dots,s$ we have that
    \[
    \max_{\ell=1,\dots,i_j} \|f|_{F_{\sigma_{j_\ell}}}\|_\infty ^{i_j-1} \int_{\mathbb R^n} \overline{\bigstar}_{\ell=1}^{i_j} F_{j_\ell}(w) dw = \frac{\prod_{\ell=1}^{i_j}|\sigma_{j_\ell}|!}{(\sum_{\ell=1}^{i_j}|\sigma_{j_\ell}|)!} \prod_{\ell=1}^{i_j}\int_{F_{\sigma_{j_\ell}}} f(x)dx.
    \]
    By the characterization of the equality case of Lemma \ref{lem:OperatorLowerBound} we have that $\|f\|_\infty=\|f|_{H_{\sigma_{i}}}\|_\infty$ for every $i=1,\dots,m$, and for every $j=1,\dots,s$ then
    $f|_{H_{\sigma_{j_\ell}}} = \|f\|_\infty \chi_{K_{j_\ell}}$,
    where $K_{j_\ell} \in \mathcal K^{|\sigma_{j_\ell}|}$ such that $0\in K_{j_\ell}$ for every $\ell=1,\dots,i_j$. Even more, we also have for every $j=1,\dots,s$ that
\begin{eqnarray*}
        f(z)&=&\overline{\bigstar}_{\ell=1}^{i_j}F_{j_\ell}(z)=\sup_{\substack{z=\sum_{\ell=1}^{i_j}\lambda_{\ell}x_{\ell}\\
         \sum_{\ell=1}^{i_j}\lambda_{\ell}=1}}\prod_{\ell=1}^{i_j}F_{j_\ell}(x_{\ell})^{\lambda_{\ell}}\\
         &=& \sup_{\substack{z=\sum_{\ell=1}^{i_j}\lambda_{\ell}x_{\ell}\\
         \sum_{\ell=1}^{i_j}\lambda_{\ell}=1}}\prod_{\ell=1}^{i_j}F_{j_\ell}(x_{\ell_{j_1}},\ldots,x_{\ell_{j_{i_j}}})^{\lambda_{\ell}}\\
         &=&\sup_{\substack{z=\sum_{\ell=1}^{i_j}\lambda_{\ell}x_{\ell}\\
         \sum_{\ell=1}^{i_j}\lambda_{\ell}=1}}\prod_{\ell=1}^{i_j}\left(f(x_{\ell_{j_\ell}})\chi_{\{0\}} (x_{\ell_{j_1}},\ldots,x_{\ell_{j_{\ell-1}}},x_{\ell_{j_{\ell+1}}},\ldots,x_{\ell_{j_{i_j}}})\right)^{\lambda_{\ell}}\\
         &=&\|f\|_\infty\sup_{\substack{z=\sum_{\ell=1}^{i_j}\lambda_\ell x_\ell\\
         \sum_{\ell=1}^{i_j}\lambda_\ell=1}}\prod_{\ell=1}^{i_j}\left(\chi_{K_{j_\ell}}(x_{\ell_{j_\ell}})\chi_{\{0\}} (x_{\ell_{j_1}},\ldots,x_{\ell_{j_{\ell-1}}},x_{\ell_{j_{\ell+1}}},\ldots,x_{\ell_{j_{i_j}}})\right)^{\lambda_\ell}\\
         &=& \|f\|_\infty\sup_{\substack{z=\sum_{\ell=1}^{i_j}\lambda_\ell x_\ell\\
         \sum_{\ell=1}^{i_j}\lambda_\ell=1}}\prod_{\ell=1}^{i_j}\chi_{ C_{j_\ell} }(x_\ell)^{\lambda_\ell}\\
         &=& \|f\|_\infty \overline{\bigstar}_{\ell=1}^{i_j} \chi_{C_{j_\ell}}(z)
\end{eqnarray*}
where $C_{j_\ell}:=\{0\}^{|\sigma_{j_1}|}\times\cdots\times\{0\}^{|\sigma_{j_{\ell-1}}|}\times K_{j_\ell}\times\{0\}^{|\sigma_{j_{\ell+1}}|}\times\cdots\times\{0\}^{|\sigma_{j_{i_j}}|}$, for every $\ell=1,\dots,i_j$. Indeed, by $(ii)$ in Proposition \ref{prop:elementaryProp} we have that
\[
\overline{\bigstar}_{\ell=1}^{i_j} \chi_{C_{j_\ell}}(z) = \chi_{C_j}(z)
\]
where $C_j:=\mathrm{conv}(\{C_{j_\ell}:\ell=1,\dots,i_j\})$, for every $j=1,\dots,s$. In particular, it holds that $\chi_{C_j}=\frac{f}{\|f\|_\infty}=\chi_{C_i}$ for every $j\neq i$, and thus $C_j$ coincide with $C_i$ for every $j\neq i$. We call that set $C$. 

Finally, we prove that
$$
C=\mathrm{conv}\left(\left\{\{0\}^{|\overline{\sigma}_1|}\times\ldots\times \{0\}^{|\overline{\sigma}_{j-1}|}\times K_j\times \{0\}^{|\overline{\sigma}_{j+1}|}\times\ldots\times \{0\}^{|\overline{\sigma}_k|}:j\in[k]\right\}\right)
$$
where $K_j\in\mathcal{K}^{|\overline{\sigma}_j|}$ for every $j=1,\dots,k$, and where $(\overline{\sigma}_1,\ldots,\overline{\sigma}_k)$ is the $1$-cover of $[n]$ induced by $(\sigma_1,\ldots,\sigma_m).$ Let us prove it by induction on $s$. 

Let us suppose that $s=2$. Then, we have for every $\ell\in\{1,\ldots,i_1\}$
\[
\begin{split}
    &K_{1_\ell}=C_1\cap H_{\sigma_{1_\ell}}=C_2\cap H_{\sigma_{1_\ell}}\\
    &=  \underset{t=1,\ldots,i_2}{\mathrm{conv}}\left(\left\{\{0\}^{|\sigma_{2_1}|}\times\ldots\times \{0\}^{|\sigma_{2_{t-1}}|}\times K_{2_t}\times \{0\}^{|\sigma_{2_{t+1}}|}\times\ldots\times \{0\}^{|\sigma_{2_{i_2}}|}\right\}\right)\cap H_{\sigma_{1_\ell}}\\
    &= \underset{t=1,\ldots,i_2}{\mathrm{conv}}\left(\left\{\{0\}^{|\sigma_{2_1}|}\times\ldots\times \{0\}^{|\sigma_{2_{t-1}}|}\times (K_{2_t}\cap H_{\sigma_{1_\ell}})\times \{0\}^{|\sigma_{2_{t+1}}|}\times\ldots\times \{0\}^{|\sigma_{2_{i_2}}|}\right\}\right),
\end{split}
\]
where we have applied Lemma \ref{lem:CapWithSubspace} in the last equality sign above. Then, that expression is clearly equal to
    \[
    \underset{ \substack{\sigma^t_{1_\ell}=\sigma_{1_\ell}\cap\sigma_{2_t}\\ t=1,\ldots,i_2} }{\mathrm{conv}} \left(\left\{\{0\}^{|\sigma_{2_1}|}\times\ldots\times \{0\}^{|\sigma_{2_{t-1}}|}\times (K_{2_t}\cap H_{\sigma^t_{1_\ell}})\times \{0\}^{|\sigma_{2_{t+1}}|}\times\ldots\times \{0\}^{|\sigma_{2_{i_2}}|} \right\}\right)
\]
where above $\sigma^t_{1_l}$ is only defined whenever $\sigma_{1_l}\cap\sigma_{2_t}\neq \emptyset$. Therefore
\[
\begin{split}
    & C_1 =\underset{\ell=1,\ldots,i_1}{\mathrm{conv}}\left(\left\{\{0\}^{|\sigma_{1_1}|}\times\ldots\times \{0\}^{|\sigma_{1_{\ell-1}}|}\times K_{1_\ell}\times \{0\}^{|\sigma_{1_{\ell+1}}|}\times\ldots\times \{0\}^{|\sigma_{1_{i_j}}|}\right\}\right)\\
    &=\mathrm{conv}\Big(\Big\{\{0\}^{|\sigma_{1_1}|}\times\ldots\times \{0\}^{|\sigma_{1_{\ell-1}}|}\times\\
    & \underset{ \substack{\sigma^t_{1_\ell}=\sigma_{1_\ell}\cap\sigma_{2_t} \\ t=1,\ldots,i_2}  }{\mathrm{conv}} \left(\left\{ \{0\}^{|\sigma_{2_1}|}\times\ldots\times \{0\}^{|\sigma_{2_{t-1}}|}\times K_{2_t}\cap H_{\sigma'_{1_t}}\times \{0\}^{|\sigma_{2_{t+1}}|}\times\ldots\times \{0\}^{|\sigma_{2_{i_2}}|}\right\}\right)\\
    &\times \{0\}^{|\sigma_{1_{\ell+1}}|}\times\ldots\times \{0\}^{|\sigma_{1_{i_1}}|}:\ell=1,\ldots,i_1\Big\}\Big)
\end{split} 
\]
where again above we only consider $\sigma^t_{1_\ell}$ such that $\sigma_{1_\ell}\cap\sigma_{2_t}\neq \emptyset$. Obviously, the last expression can be equivalently written as
\[
\underset{r=1,\dots,p}{\mathrm{conv}}\left(\left\{ \{0\}^{|\sigma^*_{1}|}\times\ldots\times \{0\}^{|\sigma^*_{r-1}|}\times C\cap H_{\sigma^*_r}\times \{0\}^{|\sigma^*_{r+1}|}\times\ldots\times \{0\}^{|\sigma^*_{p}|} \right\}\right),
\]
where $(\sigma_1^*,\dots,\sigma^*_p)$ is the $1$-cover induced by $(\sigma_{1_1},\dots,\sigma_{1_{i_1}},\sigma_{2_1},\dots,\sigma_{2_{i_2}})$. Finally, repeating this process iteratively for every $C_j$, $j=3,\dots,s$, concludes the result.
\end{proof}

We now prove the caracterization of the equality case of Theorem \ref{thm:sCoverDecom1Covers}. In order to prove it, we remember the well-known Brunn-Minkowski theorem \cite{AGM15}. Remember that the Pre\'ekopa-Leindler inequality \eqref{eq:PrekLeind} is indeed a functional counterpart of this result. First, for any given $K,C\in\mathcal K^n$, we say that $K+C:=\{x+y:x\in K,y\in C\}$ is the Minkowski sum of $K$ and $C$. Moreover, $t\cdot K:=\{tx:x\in K\}$ for every $t\geq 0$. The Brunn-Minkowski inequality states that for every $K,C\in\mathcal K^n$, and every $\lambda\in[0,1]$, it holds that
\begin{equation}\label{eq:BrunnMinkowski}
\vol{(1-\lambda)K+\lambda C}^\frac1n \geq (1-\lambda)\vol{K}^\frac1n+\lambda\vol{C}^\frac1n.
\end{equation}
Moreover, equality holds if and only either $K$ and $C$ are rescales of each other or $K$ and $C$ are contained in parallel hyperplanes.
\begin{proof}[Proof of the equality case of Theorem \ref{thm:sCoverDecom1Covers}]
    Equality holds in Theorem \ref{thm:sCoverDecom1Covers} if and only if 
    \[
    \|f\|_\infty^s \left(\int_{H_\sigma}f(x)dx\right)^{m-s}
     = \frac{\prod_{j=1}^m(|\sigma|-|\sigma_j|)!}{|\sigma|!^{m-s}} \prod_{j=1}^m\int_{H_{\sigma\setminus\sigma_j}} f(x)dx
    \]
    where $f(x)=\vol{K\cap(x+H_\sigma^\bot)}$ (see the proof of the inequality of Theorem \ref{thm:sCoverDecom1Covers}). Hence, by the equality case of Theorem \ref{thm:FunctionalRSscoverings} we would then have that $f=\|f\|_\infty\chi_{P_{H_{\sigma}}K}$, where 
    \[
    \begin{split}
      P_{H_{\sigma}}K=\underset{j=1,\dots,k}{\mathrm{conv}} \left(\left\{   \{0\}^{|\overline{\sigma}_1|}\times\cdots\times\{0\}^{|\overline{\sigma}_{j-1}|}\times K_j \times \{0\}^{|\overline{\sigma}_{j+1}|} \times\cdots \times \{0\}^{|\overline{\sigma}_k|} \right\}\right),
    \end{split}
    \]
    for some $0\in K_j\in\mathcal K^{|\overline{\sigma}_j|}$, $j=1,\dots,k$, and where $(\overline{\sigma}_1,\dots,\overline{\sigma}_k)$ is the $1$-cover induced by $(\sigma\setminus\sigma_1,\dots,\sigma\setminus\sigma_m)$. The last expression can be equivalently written as
    \[
    \begin{split}
     & P_{H_{\sigma}}K \\
     & =\underset{j=1,\dots,k}{\mathrm{conv}} \left(\left\{   \{0\}^{|\overline{\sigma}_1|}\times\cdots\times\{0\}^{|\overline{\sigma}_{j-1}|}\times (P_{H_{\sigma}}K \cap H_{\overline{\sigma}_j}) \times \{0\}^{|\overline{\sigma}_{j+1}|} \times\cdots \times \{0\}^{|\overline{\sigma}_k|} \right\}\right) \\
     & = \mathrm{conv}\left(\left\{  P_{H_{\sigma}}K \cap H_{\overline{\sigma}_j}:j=1,\dots,k   \right\}\right).
    \end{split}
    \]
    Moreover, we also have that $0\in (P_{H_{\sigma}}K)\cap H_{\overline{\sigma}_j}$, for every $j=1,\dots,k$, i.e.~$0\in P_{H_{\sigma}}K$.
    
    Now, notice that the $1$-cover of $\sigma$ induced by the $(m-s)$-cover  $(\sigma\setminus\sigma_1,\dots,\sigma\setminus\sigma_m)$ of $\sigma$ is given by $\cap_{j=1}^m(\sigma\setminus\sigma_j)^{\varepsilon(j)}$, where $\varepsilon(j)\in\{0,1\}$, and where $(\sigma\setminus\sigma_j)^0=\sigma\setminus\sigma_j$ and $(\sigma\setminus\sigma_j)^1=\sigma\setminus(\sigma\setminus\sigma_j)=\sigma_j$, i.e.~it coincides with the $1$-cover of $\sigma$ induced by the $s$-cover $(\sigma_1,\dots,\sigma_m)$ of $\sigma$.

    Finally, we go back to the condition $f=\|f\|_\infty\chi_{P_{H_\sigma}K}$. That means
    \[
    \vol{K\cap(x+H_\sigma^\bot)} = \rho = \max_{x\in H_\sigma} \vol{K\cap(x+H^\bot_\sigma)}
    \]
    for every $x\in P_{H_\sigma}K$. Notice that for every $z\in\mathrm{int}(P_{H_\sigma}K)$, we find $x,y\in \mathrm{int}(P_{H_\sigma}K)$ such that $\frac12 x+\frac12 y=z$. By the convexity of $K$ we then have that
    \[
    \frac12\left(K\cap(x+H_\sigma^\bot)\right)+\frac12\left(K\cap(y+H_\sigma^\bot)\right) \subseteq K\cap(z+H_\sigma^\bot).
    \]
    Applying $\vol{\cdot}$ to both members above, by its monotonicity and by \eqref{eq:BrunnMinkowski} we obtain that
    \[
    \begin{split}
        \rho^\frac{1}{n-|\sigma|} & = \vol{K\cap(z+H_\sigma^\bot})^\frac{1}{n-|\sigma|} \\
        & \geq \frac12 \vol{K\cap(x+H_\sigma^\bot})^\frac{1}{n-|\sigma|} +\frac12 \vol{K\cap(y+H_\sigma^\bot})^\frac{1}{n-|\sigma|} =\rho^\frac{1}{n-|\sigma|}.
    \end{split}
    \]
    We thus have equality in all inequalities above. In particular, by the equality case of \eqref{eq:BrunnMinkowski}, we then have that $K\cap(x+H_\sigma^\bot)$ and $K\cap(y+H_\sigma^\bot)$ are rescales of each other. Besides, since each of them has the same volume than the other $\rho$, then they are translates of each other. Moreover, that also means that $(z+H_\sigma^\bot)$ is a translate of them too, finishing the proof of the equality case.
\end{proof}

\section{Improved bounds in other cases}\label{sec:FurtherBounds}

We start this section by proving Theorem \ref{thm:ImprovedAlonso_Meyer}. 
\begin{proof}[Proof of Theorem \ref{thm:ImprovedAlonso_Meyer}]
    Since $K$ is an unconditional convex body, let us denote $K_+:=K\cap\mathbb R^n_+$, where $\mathbb R^n_+:=\{x\in\mathbb R^n:x_j\geq 0,\,j=1,\dots,n\}$. The idea is to show the result for $K_+$ and afterwards prove it for $K$ by means of the unconditionality.

    Let $\sigma\subset[n]$ and let $\sigma_j:=\sigma\setminus\{j\}$, for every $j\in\sigma$. Without loss of generality, we can suppose that $\sigma=\{1,\dots,|\sigma|\}$. We then have for every $x=(x_1,\dots,x_n)\in K_+$ by the convexity of $K_+$ that
    \[
    K_+ \supseteq \bigcup_{j\in\sigma} \mathrm{conv}\left(\{x\}\cup (K_+\cap(H_{\sigma_j}\oplus H_\sigma^\bot))\right).
    \]
    Notice that since each set in the union above is contained in one of the cones $\{x\in\mathbb R^n:x_j\geq 0,\,j=1,\dots,n\} \cap \langle e_i\rangle^\bot$, $i\in\sigma$, then such union is formed by sets whose intersection has measure $0$. Therefore, taking volumes above and using the monotonicity we get that
    \[
    \begin{split}
    \vol{K_+} & \geq \sum_{i\in\sigma} \vol{\mathrm{conv}\left( \{x\}\cup(K_+\cap(H_{\sigma_i}\oplus H_{\sigma}^\bot)) \right)} \\
    & = \sum_{i\in\sigma} \frac1n x_i \vol{K_+\cap(H_{\sigma_i}\oplus H_{\sigma}^\bot)}.
        \end{split}
    \]
    Moreover, since $x\in K_+$, by the unconditionality of $K$ then 
    \[
    (x_1,\dots,x_{|\sigma|},x_{|\sigma|+1},\dots,x_n),(-x_1,\dots,-x_{|\sigma|},x_{|\sigma|+1},\dots,x_n) \in K,
    \]
    and by the convexity of $K$ then $(0,\dots,0,x_{|\sigma|+1},\dots,x_n)\in K$, and thus $(x_{|\sigma|+1},\dots,x_n)\in K_+\cap \langle \{e_{|\sigma|+1},\dots,e_n\}\rangle$.
    This means that
    \[
    \begin{split}
         K_+ \subseteq \Big\{  x\in\mathbb R^n_+:\vol{K_+} \geq \sum_{j\in\sigma} \frac1n x_j \vol{K_+\cap(H_{\sigma_j}\oplus H_{\sigma}^\bot)}, & \\
         (x_{|\sigma|+1},\dots,x_n)\in K_+\cap \langle \{e_{|\sigma|+1},\dots,e_n\}\rangle  \Big\} &
    \end{split}
    \]
    Taking volumes above and by its monotonicity we conclude that
    \[
    \begin{split}
        &\vol{K_+} \leq \mathrm{vol}\Big( \Big\{  x\in\mathbb R^n_+:\vol{K_+} \geq \sum_{j\in\sigma} \frac1n x_j \vol{K_+\cap(H_{\sigma_j}\oplus H_{\sigma}^\bot)}, \\
        & \hspace{4.1cm} (x_{|\sigma|+1},\dots,x_n)\in K_+\cap \langle \{e_{|\sigma|+1},\dots,e_n\}\rangle  \Big\} \Big) \\
        & = \mathrm{vol} \Big[ \mathrm{diag}\left(\rho_1,\dots,\rho_{|\sigma|}\right) \{x\in\mathbb R^{|\sigma|}_+\cap H_\sigma:\sum_{j\in\sigma}x_j\leq 1\} \Big] \cdot\vol{K_+\cap \langle \{e_{|\sigma|+1},\dots,e_n\}\rangle} 
    \end{split}
    \]
    where $\rho_j:=\frac{\vol{K_+}}{\vol{K_+\cap(H_{\sigma_j}\oplus H_{\sigma}^\bot)}}$, for every $j\in\sigma$. Therefore, we obtain that
    \[
    \vol{K_+} \leq \frac{n^{|\sigma|}\vol{K_+}^{|\sigma|}}{\prod_{j\in\sigma}\vol{K_+\cap(H_{\sigma_j}\oplus H_\sigma^\bot)}} \frac{1}{|\sigma|!} \vol{K_+\cap \langle \{e_{|\sigma|+1},\dots,e_n\}\rangle},
    \]
    and the result follows for $K_+$. Using again the unconditionality of $K$, we can once again derive the result for $K$.
\end{proof}

We now observe that the inequality derived in Theorem \ref{thm:ImprovedAlonso_Meyer} is better than the one obtained in \eqref{eq:IneqOriginalABBC}. We quantify this improvement in the next proposition by just writing $|\sigma|=p$ and supposing that $n\geq 4|\sigma|$. We postpone its proof to the Section \ref{sec:Appendix} (Appendix).
\begin{propo}\label{prop:QuantitativeEstim1}
    Let 
     $2\leq p\leq \frac{n}{4}$. Then we have that
    $$\dfrac{\dfrac{p!}{n^{p}}}{\dfrac{(p-1)!^{p-1}}{(n(p-1))!}}\geq
      n^{\frac{n(p-2)}{4}} \cdot (p-1)^{\frac{(n-4p)(p-1)}{2}} \cdot (p-1)^{(p-1)(p+2)}.$$
\end{propo}

Proposition \ref{prop:QuantitativeEstim1} cares about the behavior of the quotient between the constants in Theorem \ref{thm:ImprovedAlonso_Meyer} and \eqref{eq:IneqOriginalABBC} for large values of $n$, at least as large as $n\geq 4|\sigma|$. Computational experiments suggest that the constant in the former is actually always better than in the latter. In order to highlight this, we now compare the constants in the most different case to the ones covered in Proposition \ref{prop:QuantitativeEstim1} of $n=|\sigma|+1$. Again, its proof is postponed to Section \ref{sec:Appendix}.
\begin{propo}\label{prop:quantitatiVEstim2}
    For every $n \geq 5$, we have that
    \[
    \frac{\frac{(n-1)!}{n^{n-1}}}{\frac{(n-2)!^{n-2}}{(n(n-2))!}} \geq 2^{n-5}(n-2)^{n-4}(n(n-2))^{\frac{n(n-2)}{16}}.
    \]
\end{propo}

Let us observe that even though Proposition \ref{prop:quantitatiVEstim2} skips the case of $n=4$, the constant achieved in Theorem \ref{thm:ImprovedAlonso_Meyer} is surely better than the one achieved in \eqref{eq:IneqOriginalABBC} when $n=|\sigma|+1$, since in that case we have that the quotient of the constants equals
\[
\frac{\frac{(n-1)!}{n^{n-1}}}{\frac{(n-2)!^{n-2}}{(n(n-2))!}} = \frac{\frac{3!}{4^3}}{\frac{2!^{2}}{(4\cdot 2)!}}= 945.
\]

As a final result, we show that if we already assume that $\vol{K\cap(x+H_\sigma^\bot)}$ is constant for every $x\in P_{H_\sigma}K$ and $K$ is unconditional, then we can actually show Conjecture \ref{conj:scover}.
\begin{propo}
    Let $K\in\mathcal K^n$ and let 
    $(\sigma_1,\dots,\sigma_m)$ a $s$-cover of $\sigma\subset[n]$. If $K$ is unconditional and $\vol{K\cap(x+H_\sigma^\bot)}$ is constant for every $x\in K\cap H^\bot_\sigma$, then Conjecture \ref{conj:scover} holds for $K$, $\sigma$, $\sigma_1$,\dots,$\sigma_m$.
\end{propo}

\begin{proof}
    Let us observe that $(\sigma\setminus\sigma_1,\dots,\sigma\setminus\sigma_m)$ is a $(m-s)$-cover of $\sigma$. Notice also that due to the unconditionality of $K$ then
    \[
    \vol{K} = \int_{K\cap H_\sigma}\vol{K\cap(x+H_\sigma^\bot)}dx = \vol{K\cap(x+H_\sigma^\bot)}\vol{K\cap H_\sigma}
    \]
    and
    \[
    \begin{split}
    \vol{K\cap H_{\sigma_j}^\bot} & = \int_{K\cap H_{\sigma\setminus\sigma_j}}\vol{K\cap(x+H_\sigma^\bot)}dx \\
    & = \vol{K\cap H_\sigma^\bot} \vol{K\cap H_{\sigma\setminus\sigma_j}}.
    \end{split}
    \]
    Applying \eqref{eq:Liako19} to $K\cap H_\sigma$ and to the $(m-s)$-cover $(\sigma\setminus\sigma_1,\dots,\sigma\setminus\sigma_m)$ we get that
    \[
    \begin{split}
    \vol{K}^{m-s} & = \vol{K\cap(x+H_\sigma^\bot)}^{m-s}\vol{K\cap H_\sigma}^{m-s} \\
    & \geq \vol{K\cap(x+H_\sigma^\bot)}^{m-s} \frac{\prod_{j=1}^m(|\sigma|-|\sigma_j|)!}{|\sigma|!^{m-s}} \prod_{j=1}^m \vol{K\cap H_{\sigma\setminus\sigma_j}}
    \end{split}
    \]
    and thus
    \[
    \begin{split}
        \vol{K}^{m-s}\vol{K\cap H_\sigma^\bot}^{s} & \geq \frac{\prod_{j=1}^m(|\sigma|-|\sigma_j|)!}{|\sigma|!^{m-s}} \prod_{j=1}^m \Big[\vol{K\cap H_{\sigma\setminus\sigma_j}}\vol{K\cap H_\sigma^\bot}\Big] \\
        & = \frac{\prod_{j=1}^m(|\sigma|-|\sigma_j|)!}{|\sigma|!^{m-s}} \prod_{j=1}^m \vol{K\cap H_{\sigma_j}^\bot},
    \end{split}
    \]
    concluding the proof.
\end{proof}

\section{Appendix}\label{sec:Appendix}

\begin{proof}[Proof of Proposition \ref{prop:QuantitativeEstim1}]
    Let us observe that the factorial is lower bounded by $k! \geq \left(\frac k2\right)^\frac k2$ for every $k\in\mathbb N$, then
\begin{equation}\label{eq:quantitative1}
\begin{split}
    \dfrac{\dfrac{p!}{n^{p}}}{\dfrac{(p-1)!^{p-1}}{(n(p-1))!}} & =\dfrac{p!(n(p-1))!}{n^{p}(p-1)!^{p-1}} \geq \dfrac{\left(\frac{n}{2}\right)^{\frac{n(p-1)}{2}}}{n^p}\dfrac{p!(p-1)^{\frac{n(p-1)}{2}}}{(p-1)!^{p-1}} \\
    & \geq\dfrac{\left(\frac{n}{2}\right)^{\frac{n(p-1)}{2}}}{n^p} \cdot \dfrac{(p-1)^{\frac{n(p-1)}{2}}}{(p-1)!^{p-2}}.
\end{split}
\end{equation}

On the one hand, the left-hand side of the last line in \eqref{eq:quantitative1} can be rewritten as
\begin{eqnarray*}
    \dfrac{\left(\frac{n}{2}\right)^{\frac{n(p-1)}{2}}}{n^p}&=&\left(\frac{n}{2}\right)^{\frac{n(p-1)}{4}}\left(\frac{n}{2}\right)^{\frac{n(p-1)}{4}}\frac{1}{n^p}=\left(\frac{n}{2}\right)^{\frac{n(p-1)}{4}}\left(\frac{\left(\frac{n}{2}\right)^\frac{n}{4}}{n}\right)^p\frac{1}{\left(\frac{n}{2}\right)^\frac{n}{4}}\\
    &=&\left(\frac{n}{2}\right)^{\frac{n(p-2)}{4}}\left(\frac{\left(\frac{n}{2}\right)^\frac{n}{4}}{n}\right)^p=\left(\frac{n}{2}\right)^{\frac{n(p-2)}{4}}\left(\frac{1}{2}\left(\frac{n}{2}\right)^{\frac{n}{4}-1}\right)^p.
\end{eqnarray*}
Since $\frac{n}{4}\geq1$, we thus get that
\begin{eqnarray*}
    \left(\frac{n}{2}\right)^{\frac{n(p-2)}{4}}\left(\frac{1}{2}\left(\frac{n}{2}\right)^{\frac{n}{4}-1}\right)^p&\geq&\left(\frac{n}{2}\right)^{\frac{n(p-2)}{4}}\left(\frac{n}{2}\right)^{p\left(\frac{n}{4}-2\right)}.
\end{eqnarray*}
Finally, using the fact that $\frac{n}{2}\geq\sqrt{n}$ whenever $n\geq 4$, and since $n\geq 4p$ then 
\begin{eqnarray*}
    \left(\frac{n}{2}\right)^{\frac{n(p-2)}{4}}\left(\frac{n}{2}\right)^{p\left(\frac{n}{4}-2\right)}
    = \left(\frac{n}{2}\right)^{\frac{np-n-4p}{2}} \geq \left(\frac{n}{2}\right)^{\frac{n(p-2)}{2}}\geq n^{\frac{n(p-2)}{4}}.
\end{eqnarray*}

On the other hand, taking into account that $k!\leq k^k$ for every $k\in\mathbb N$, the right-hand side of the last line in \eqref{eq:quantitative1} fulfills
\begin{eqnarray*}
    \dfrac{(p-1)^{\frac{n(p-1)}{2}}}{(p-1)!^{p-2}}&=&\dfrac{(p-1)^{\frac{4p(p-1)}{2}}}{(p-1)!^{p-2}}(p-1)^{\frac{(n-4p)(p-1)}{2}} \\
    &\geq & \dfrac{(p-1)^{\frac{4p(p-1)}{2}}}{(p-1)^{(p-1)(p-2)}}(p-1)^{\frac{(n-4p)(p-1)}{2}}\\
    &=&(p-1)^{(p-1)(p+2)}(p-1)^{\frac{(n-4p)(p-1)}{2}}.
\end{eqnarray*}
Glueing all together within \eqref{eq:quantitative1}, we conclude that
\begin{eqnarray*}
    \dfrac{\dfrac{p!}{n^{p}}}{\dfrac{(p-1)!^{p-1}}{(n(p-1))!}}\geq n^{\frac{n(p-2)}{4}}(p-1)^{\frac{(n-4p)(p-1)}{2}}(p-1)^{(p-1)(p+2)}.\qedhere
\end{eqnarray*}
\end{proof}

\begin{proof}[Proof of Proposition \ref{prop:quantitatiVEstim2}]
    We first remember that 
    \[
    k! \geq \left(\frac k2\right)^{\frac k2} \left(\frac k4\right)^{\frac k4}
    \]
    for every $k\in\mathbb N$. We thus get that
    \begin{equation}\label{eq:Quant2}
    \frac{(n(n-2))!}{(n-2)!^{n-2}} \geq \frac{(n-2)^{\frac{n(n-2)}{2}}}{(n-2)!^{n-2}} \left(\frac n2\right)^{\frac{n(n-2)}{2}} \left(\frac{n(n-2)}{4}\right)^{\frac{n(n-2)}{4}}.
    \end{equation}
    We now remember that 
    \[
    k!\leq k^{\frac k2} \left(\frac k2\right)^{\frac{k}{2}}
    \]
    for every $k\in\mathbb N$. Now let us notice that 
    \[
    \frac{(n-2)^{\frac{n}{2}}}{(n-2)!} \geq \frac{(n-2)^{\frac{n}{2}}}{(n-2)^{\frac{n-2}{2}}\left(\frac{n-2}{2}\right)^{\frac{n-2}{2}}} = (n-2)^{-\frac n2+2} \cdot 2^{\frac{n-2}{2}}. 
    \]
    Inserting the above within \eqref{eq:Quant2} we obtain that the expression is greater than or equal to
    \[
    \begin{split}
     & (n-2)^{(n-2)(-\frac n2+2)} 2^{\frac{(n-2)^2}{2}} \left(\frac n2\right)^{\frac{n(n-2)}{2}} \left(\frac{n(n-2)}{4}\right)^{\frac{n(n-2)}{4}} \\
     & \geq (n-2)^{\frac{n(n-2)}{2}+(n-2)(-\frac n2+2)} \left(\frac{n(n-2)}{4}\right)^{\frac{n(n-2)}{4}} 2^{\frac{(n-2)^2}{2}-\frac{n(n-2)}{2}} \\
     & = (n-2)^{2n-4} \left(\frac{n(n-2)}{4}\right)^{\frac{n(n-2)}{4}} 2^{2-n}.
    \end{split}
    \]
    Thus, we get that
    \[
    \begin{split}
    \frac{(n-1)!(n(n-2))!}{n^{n-1}(n-2)!^{n-2}} & \geq \frac{(n-1)!}{n^{n-1}} (n-2)^{2n-4} \left(\frac{n(n-2)}{4}\right)^{\frac{n(n-2)}{4}} \frac{1}{2^{n-2}} \\
    & = \frac{(n-1)!}{2^{n-2}} \frac{(n-2)^{2n-4}}{n^{n-1}} \cdot \left(\frac{n(n-2)}{4}\right)^{\frac{n(n-2)}{4}}.
    \end{split}
    \]
    On the one hand, since $n(n-2) \geq 4$, the right-hand side above is greater than or equal to
    \[
    \left(n(n-2)\right)^{\frac{n(n-2)}{16}}.
    \]
    On the other hand, since 
    \[
    (n-1)!=(n-1)\cdots 4\cdot 3\cdot 2\cdot 1 \geq 4^{n-4}\cdot 3\cdot 2
    \]
    and since
    \[
    \begin{split}
    \frac{(n-2)^{2n-4}}{n^{n-1}} & = \left(\frac{n-2}{n}\right)^{n-1} (n-2)^{n-3} = \left(1-\frac 2n\right)^{n-1} (n-2)^{n-3} \\
    & \geq e^{-2} (n-2)^{n-3},
    \end{split}
    \]
    we thus get on the left-hand side that
    \[
    \begin{split}
    \frac{(n-1)!}{2^{n-2}} \frac{(n-2)^{2n-4}}{n^{n-1}}  & \geq \frac{4^{n-4}\cdot 3\cdot 2}{2^{n-2}} e^{-2}(n-2)^{n-3} \\
    & = 2^{n-5}\frac{3(n-2)}{e^2}(n-2)^{n-4} \\
    & \geq 2^{n-5}(n-2)^{n-4},
    \end{split}
    \]
    where the last inequality holds because $3(n-2)e^{-2} \geq 1$ for every $n\geq 5$. Thus we can conclude that
    \[
    \frac{(n-1)!(n(n-2))!}{n^{n-1}(n-2)!^{n-2}} \geq 2^{n-5}(n-2)^{n-4}(n(n-2))^{\frac{n(n-2)}{16}}
    \]
    for every $n\geq 5$, finishing the proof.
\end{proof}


\begin{thebibliography}{XXXX}

\bibitem{A19}{\sc Alonso‐Guti\'errez, D.}
\textit{A reverse Rogers–Shephard inequality for log-concave functions.}
J. Geom. Anal., {\bf 29}  (2019), no. 1, 299--315.

\bibitem{AAGJV}{\sc Alonso-Guti\'errez, D., Artstein-Avidan, S., Gonz\'alez Merino, B., Jim\'enez, C. H., Villa, R.}
\textit{Rogers-Shephard and local Loomis-Whitney type inequalities.}
Math. Annalen, {\bf 374} (2019), no. 3-4, 1719--1771.

\bibitem{ABBC21}{\sc Alonso‐Guti\'errez, D., Bernu\'es, J., Brazitikos, S., Carbery, A.}
\textit{On affine invariant and local Loomis–Whitney type inequalities.}
J. Lond. Math. Soc., {\bf 103}  (2021), no. 4, 1377--1401.

\bibitem{ABG20}{\sc Alonso‐Guti\'errez, D., Bernu\'es, J., González Merino, B.}
\textit{Zhang’s inequality for log-concave functions.}
In Geometric Aspects of Functional Analysis: Israel Seminar (GAFA) 2017-2019. Cham: Springer International Publishing, 2020. p. 29--48.


\bibitem{ABG202}{\sc Alonso‐Guti\'errez, D., Bernu\'es, J., González Merino, B.}
\textit{An extension of Berwald's inequality and its relation to Zhang's inequality.}
J. Math. Anal. Appl., {\bf 486}  (2020), no. 1, p. 123875.





\bibitem{AGJV16}{\sc Alonso‐Guti\'errez, D., González Merino, B., Jim\'enez, C. H., Villa, R}
\textit{Rogers–Shephard inequality for log-concave functions.}
J. Funct. Anal., {\bf 271}  (2016), no. 11, 3269--3299.


\bibitem{AGJV18}{\sc Alonso‐Guti\'errez, D., González Merino, B., Jim\'enez, C. H., Villa, R}
\textit{John’s ellipsoid and the integral ratio of a log-concave function.}
J. Geom. Anal., {\bf 28}  (2018), no. 2, 1182--1201.




\bibitem{AGM15}{\sc Artstein-Avidan, S., Giannopoulos, A., Milman, V. D.}
\textit{Asymptotic Geometric Analysis, Part I.}
Amer. Math. Soc. (2015), Vol. 202. 


\bibitem{AKM04}{\sc Artstein-Avidan, S., Klartag, B., Milman, V. D.}
\textit{The Santal\'o point of a function, and a functional form of the Santal\'o inequality.}
Mathematika, {\bf 51}  (2004), no. 1--2, 33--48.








\bibitem{B88}{\sc Ball, K.}
\textit{Logarithmically concave functions and sections of convex sets in $\mathbb{R}^n$.} 
Studia Math, {\bf 88} (1988), no. 1, p. 69--84.






\bibitem{BCF14}{\sc Bobkov, S. G., Colesanti, A., Fragal\`a, I.}
\textit{Quermassintegrals of quasi-concave functions and generalized Pr\'ekopa–Leindler inequalities.} 
Manuscripta Math., {\bf 143} (2014), no. 1, p. 131---169.



\bibitem{BT95}{\sc Bollob\'as, T., Thomason, A.} 
\textit{Projections of bodies and hereditary properties of hypergraphs.}
Bull. London Math. Soc. {\bf 27} (1995), 417--424.

\bibitem{BKX23}{\sc Boroczky, K. J., Kalantzopoulos, P., Xi, D.}
\textit{The case of equality in geometric instances of Barthe’s reverse Brascamp-Lieb inequality.}
In Geometric Aspects of Functional Analysis: Israel Seminar (GAFA) 2020-2022. Cham: Springer International Publishing, 2023. p. 129--165.

\bibitem{BGL18}{\sc Brazitikos, S., Giannopoulos, A. Liakopoulos, D. M.}
\textit{Uniform cover inequalities for the volume of coordinate sections and projections of convex bodies.} 
Adv. Geom., {\bf 18} (2018), no. 3, p. 345--354.

\bibitem{CF13}{\sc Colesanti, A., Fragal\`a, I.}
\textit{The first variation of the total mass of log-concave functions and related inequalities.}
Adv. Math. {\bf 244} (2013), 708--749.



\bibitem{Co06}{\sc Colesanti, A.}
\textit{Functional inequalities related to the Rogers-Shephard inequality.}
Mathematika {\bf 53} (2006), no. 1, 81--101.

\bibitem{FSY22}{\sc Fang, N., Sudan, Y., Ye, D.}
\textit{Geometry of log-concave functions: the $L_p$ Asplund sum and the $L_p$ Minkowski problem.}
 Calc. Var.
Partial Differential Equations {\bf 61} (2022), no. 2, p. 45.


\bibitem{FZ18}{\sc Fang, N., Zhou, J.}
\textit{LYZ ellipsoid and Petty projection body for log-concave functions.}
Adv. Math. {\bf 340} (2018), 914--959.





\bibitem{FM07}{\sc Fradelizi, M., Meyer, M.}
\textit{Some functional forms of Blaschke–Santal\'o inequality.}
 Math. Z. {\bf 256} (2007), no. 2, 379--395.

\bibitem{FM08}{\sc Fradelizi, M., Meyer, M.}
\textit{Some functional inverse Santal\'o inequalities.}
 Adv. Math. {\bf 218} (2008), no. 5, 1430--1452.
 
 


\bibitem{GHP02}{\sc Giannopoulos, A., Hartzoulaki, M., Paouris, G.} 
\textit{On a local version of the Aleksandrov-Fenchel inequality for the quermassintegrals of a convex body.}
Proc. A. M. S., {\bf 130} (2002), no. 8, 2403--2412.

\bibitem{IN22}{\sc Ivanov, G., Nasz\'odi, M.}
\textit{Functional John ellipsoids.} 
J. Funct. Anal., {\bf 282} (2022), no. 11, 109441.

\bibitem{IT21}{\sc Ivanov, G., Tsiutsiurupa, I.}
\textit{Functional L\"owner Ellipsoids.} 
J. Geom. Anal., {\bf 31} (2021), no. 11, 11493--11528.

\bibitem{KM05}{\sc Klartag, B., Milman, V. D.}
\textit{Geometry of log-concave functions and measures.} 
Geom. Dedicata {\bf 112} (2005), no. 1, 169--182.

\bibitem{LMU24}{\sc Langharst, D., Marín Sola, F., Ulivelli, J.}
\textit{Higher-order reverse isoperimetric inequalities for log-concave functions.} 
arXiv:2403.05712.


\bibitem{LSW21}{\sc Li, B., Schuütt, C.. Werner, E.}  
\textit{The Löwner function of a log-concave function.}
J. Geom.
Anal., {\bf 31} (2021), no. 1, 423--456.

\bibitem{Le72}{\sc Leindler, L.}  
\textit{On a certain converse of H\"older’s inequality.}
In Linear Operators and Approximation: Proceedings of the Conference Held at the Oberwolfach Mathematical Research Institute, Black Forest, August 14--22, 1971, page 182. Birkhauser, 1972.

\bibitem{L19}{\sc Liakopoulos, D. M.}  
\textit{Reverse Brascamp–Lieb inequality and the dual Bollob\'as–Thomason inequality.}
Arch. Math., {\bf 112} (2019), no. 3, 293--304.

\bibitem{LM49}{\sc Loomis, L. H., Whitney, H.}
\textit{An inequality related to the isoperimetric inequality.}
Bull. Amer. Math. Soc. {\bf 55} (1949), 961--962.

\bibitem{MNZ24}{\sc Manui, A., Ndiaye, C. S., Zvavitch, A.} 
\textit{On the volume of sums of anti-blocking bodies.}
arXiv:2409.14214v1. 


\bibitem{M88}{\sc Meyer, M.}
\textit{A volume inequality concerning sections of convex sets.}
Bull. London Math. Soc. {\bf 20} (1988), 151--155.

\bibitem{Pr71}{Pr\'ekopa, A.}
\textit{Logarithmic concave measures with application to stochastic programming.}
Acta Sci. Math., {\bf 32} (1971), 301--315.

\bibitem{RS58}{\sc Rogers, C. A., Shephard, G. C.}
\textit{Convex bodies associated with a given convex body.}
J. Lond. Math. Soc. {\bf 1} (1958), no. 3, 270--281.

\bibitem{T25}{\sc Tziotziou, N.}
\textit{Inequalities for sections and projections of log-concave functions.}
arXiv:2503.21494.



\end{thebibliography}
\end{document}